\documentclass[hidelinks, 12pt, reqno]{amsart}

\usepackage{amsmath, bm}
\usepackage[psamsfonts]{amssymb}
\usepackage{amsfonts}
\usepackage{algorithm2e}
\usepackage{graphicx}
\usepackage{verbatim}
\usepackage{dsfont}
\usepackage{tikz-cd}
\usepackage{mathrsfs}
\usepackage{color}
\usepackage[english]{babel}
\usepackage{fontenc}
\usepackage{indentfirst}
\usepackage{tikz}
\usetikzlibrary{matrix,arrows,decorations.pathmorphing}
\usepackage{algorithm2e}
\usepackage[all]{xy}
\usepackage[T1]{fontenc}
\usepackage{hyperref}
\usepackage{mathtools}
\usepackage{multicol}
\usepackage{float}
\usepackage{physics}
\usepackage{dsfont}
\usepackage{multirow}
\usetikzlibrary{tikzmark}
\usepackage{caption}
\usepackage{float}

\newtheorem{theorem}{Theorem}[section]
\newtheorem{prop}[theorem]{Proposition}
\newtheorem{lemma}[theorem]{Lemma}
\newtheorem{cor}[theorem]{Corollary}
\newtheorem{assm}[theorem]{Assumption}

\newtheorem{conj}[theorem]{Conjecture}

\newtheorem{ex}[theorem]{Example}
\newtheorem{dfn}[theorem]{Definition}
\newtheorem{remark}[theorem]{Remark}

\newtheorem{thm}{Theorem}

\newcommand{\bC}{\mathbb{C}}
\newcommand{\bD}{\mathbb{D}}

\newcommand{\bN}{\mathbb{N}}

\newcommand{\bR}{\mathbb{R}}

\newcommand{\bT}{\mathbb{T}}

\newcommand{\bZ}{\mathbb{Z}}

\newcommand{\cA}{\mathcal{A}}

\newcommand{\cD}{\mathcal{D}}

\newcommand{\cG}{\mathcal{G}}
\newcommand{\cH}{\mathcal{H}}

\newcommand{\cJ}{\mathcal{J}}

\newcommand{\cL}{\mathcal{L}}
\newcommand{\cM}{\mathcal{M}}

\newcommand{\cP}{\mathcal{P}}

\newcommand{\fc}{\mathfrak{c}}

\newcommand{\ra}{\rightarrow}

\newcommand{\mL}{\mathds{1}_L}

\newcommand{\mW}{\mathds{1}_W}

\DeclareMathOperator{\ham}{Ham}
\DeclareMathOperator{\ind}{Ind}
\DeclareMathOperator{\spec}{Spec}

\DeclareMathOperator{\CZ}{CZ}
\DeclareMathOperator{\CW}{CW}
\DeclareMathOperator{\CF}{CF}
\DeclareMathOperator{\HW}{HW}
\DeclareMathOperator{\HF}{HF}
\DeclareMathOperator{\minmax}{minmax}

\allowdisplaybreaks

\begin{document}
\title{Spectrally-large scale geometry in cotangent bundles}

\author{Qi Feng}
\email{feqi@mail.ustc.edu.cn}
\address{School of Mathematical Sciences, University of Science and Technology of China, 96 Jinzhai Road, Hefei Anhui, 230026, China}

\author{Jun Zhang}
\email{jzhang4518@ustc.edu.cn}
\address{The Institute of Geometry and Physics, University of Science and Technology of China, 96 Jinzhai Road, Hefei Anhui, 230026, China}

\begin{abstract}
    In this paper, we prove that the ${\rm Ham}$-orbit space from a fiber of a large family of cotangent bundles, as a metric space with respect to the Floer-theoretic spectral metric, contains a quasi-isometric embedding of an infinite-dimensional normed vector space. The same conclusion holds for the group of compactly supported Hamiltonian diffeomorphisms of some cotangent bundles. To prove this, we generalize a result, relating boundary depth and spectral norm for closed symplectic manifolds in Kislev-Shelukhin \cite{KS21}, to Liouville domains. Then we modify Usher's constructions in \cite{Ush13,Ush14} (which were used to obtain Hofer-large scale geometric properties) to achieve our desired conclusions.
\end{abstract}
\maketitle
\section{Introduction}
\subsection{Main results} Let $(M,\omega)$ be a compact symplectic manifold (possibly with boundary) and let $H:[0,1]\times M\ra \bR$ be a smooth function (called a Hamiltonian), which is compactly supported in $[0,1]\times {\rm int}(M)$. Hamiltonian $H$ generates a flow $\phi_H^t$ by integrating the Hamiltonian vector field $X_{H_t}$ determined by $\omega(\cdot , X_{H_t})=dH_t$. We denote $\ham(M,\omega)$ as the group of time-one maps of the flow $\phi^t_H$, called the Hamiltonian diffeomorphism group of $(M, \omega)$. There have been several interesting metrics defined on ${\rm Ham}(M, \omega)$. For instance, due to \cite{Hof90}, for any $\phi\in\ham(M,\omega)$, consider 
$$\|\phi\|_{\rm Hofer}\coloneqq \inf\left\{\int_0^1 \max_{M}H(t,\cdot)-\min_{M}H(t,\cdot)\,dt\,\bigg|\,\phi^1_H=\phi\right\},$$
then $d_{\rm{Hofer}}(\phi,\psi)=\|\phi^{-1}\psi\|_{\rm Hofer}$ defines a Finsler-type metric on ${\rm Ham}(M, \omega)$. To describe large-scale geometric properties of the metric space $({\rm Ham}(M, \omega), d_{\rm Hofer})$ in a convenient way, let us introduce the following concept. 

\begin{dfn} 
Let $(M_1,d_1), (M_2, d_2)$ be metric spaces. A map $f: M_1 \to M_2$ is called a {\rm quasi-isometric embedding} if there exist constant $A\geq 1, B \geq 0$ such that 
    $$\dfrac{1}{A}d_1(x,y)-B\leq d_2(f(x),f(y))\leq Ad_1(x,y)+B$$
for any $x,y\in M_1$. We say that $(M,d)$ contains a {\rm rank-$n$ quasi-flat} if there is a quasi-isometric embedding from $(\mathbb{R}^n,d_\infty)$ to $(M,d)$, where $n\in\bN\cup\{\infty\}$ and $d_{\infty}(x,y)=\|x-y\|_\infty$. 
\end{dfn}

A result in \cite{Ush13} shows that metric space $({\rm Ham}(M, \omega), d_{\rm Hofer})$ contains a rank-$\infty$ quasi-flat if $(M, \omega)$ is a closed symplectic manifold which satisfies some dynamical property. 

\medskip

In this paper, we consider another norm derived from Hamiltonian Floer theory\footnote{In this paper, all the cohomology groups are in $\bZ_2$-coefficient. }. For a Liouville domain $(W,\omega)$, one associates spectral invariant $c(\alpha,H)$ to any pair $(\alpha,H)\in H^*(W) \times C_c^\infty([0,1]\times W)$, which belongs to the spectrum of the action functional of $H$. In \cite{FS07}, it shows that for compactly supported Hamiltonians $H$ and $F$ with the same time-one map $\phi_H^1 = \phi_F^1$, the spectral invariant $c(\alpha, H)=c(\alpha, F)$ for any $\alpha\in H^*(W)$.  Therefore, the spectral invariant can be defined on the group of compactly supported Hamiltonian diffeomorphisms $\ham(W,\omega)$.  For any $\phi = \phi_H^1 \in \ham(W,\omega)$, denote $c(\alpha,\phi)$ by $c(\alpha,H)$ and define the spectral norm $\gamma$ on $\ham(W,\omega)$ as follows,
\begin{equation} \label{gamma-phi}
\gamma(\phi)\coloneqq c(\mW,H)+c(\mW,\overline{H})
\end{equation}
where $\mW$ is the unit in $H^*(W)$ and $\overline{H}$ generates $\phi^{-1}$. Roughly speaking, $\gamma(\phi)$ measures the largest gap between the spectrum of the action functional of $H$. Then for any $\phi,\psi\in\ham(W,\omega)$, define $d_{\gamma}(\phi,\psi)\coloneqq\gamma(\phi^{-1}\psi)$, which turns out to be a bi-invariant metric on $\ham(W,\omega)$. It is readily verified that 
\begin{equation}\label{gamma-hofer}
\gamma(\phi) \leq \|\phi\|_{\rm Hofer}.
\end{equation} Therefore, the large-scale geometric phenomenon with respect to $d_{\rm Hofer}$ can not imply the same phenomenon with respect to $d_{\gamma}$. 

\medskip

The following theorem describes the large-scale geometric phenomenon in the metric space $\left(\ham(D^*_gN,\omega_{\rm{can}}),d_\gamma\right)$.
\begin{thm}\label{Ham-infinite-flat}
   Let $(N, g)$ be a closed Riemannian manifold. Suppose there is no non-constant  contractible  closed  geodesic in $(N,g)$, then metric space $\left(\ham(D_g^*N,\omega_{\rm can}),d_{\gamma}\right)$ contains a rank-$\infty$ quasi-flat. 
\end{thm}
 
Theorem~\ref{Ham-infinite-flat} above generalizes the main result in \cite{Mai22} when we specify our setting on the symplectic manifold $(D^*_gN, \omega_{\rm can})$, where \cite{Mai22} proves the unboundedness of $\left(\ham(D_g^*N,\omega_{\rm can}),d_{\gamma}\right)$. Standard manifolds $N$ where Theorem~\ref{Ham-infinite-flat} applies include $\bT^n$, closed hyperbolic manifolds, as well as their products. In particular, $N =S^n$ for $n \geq 2$ are excluded (cf.~Corollary \ref{cor-sn}). Note that \cite{Mai22} in fact proves the unboundedness of $\left(\ham(W,\omega), d_\gamma\right)$ for any Liouville domain $(W, \omega)$ whenever ${\rm SH}^*(W, \omega) \neq 0$ (where $W = D^*_gN$ satisfies this condition due to Viterbo's theorem \cite{Vit92}). Inspired by Theorem~\ref{Ham-infinite-flat}, we post the following conjecture, where its proof might depend on the Morse (co)homology of the ``core'' (or sometimes called the skeleton) of a Liouville domain $(W, \omega)$.

\begin{conj} \label{conj-1} For any Liouville domain $(W, \omega)$ with ${\rm SH}^*(W, \omega) \neq 0$, the metric space $\left(\ham(W,\omega),d_{\gamma}\right)$ contains a rank-$\infty$ quasi-flat. \end{conj}

There is an ongoing work \cite{FZ25} proving that if $2c_1(W)=0$, then ${\rm SH}^*(W,\omega)\neq 0$ if and only if  there exists  a quasi-isometric embedding from $(C^\infty_c(I), d_{\infty})$ to $(\ham(W,\omega), d_{\gamma})$.

\medskip

Next, let us turn to the relative case, where we consider Lagrangian submanifolds $L \subset M$ (where $L$ could intersect $\partial M$). Denote by $\mathcal{L}(L)$ the orbit space of a Lagrangian manifold $L$ under the action by group $\ham(M,\omega)$, that is, 
\begin{equation} \label{Lag-orb}
\mathcal L(L) : = \{\phi(L) \,| \, \phi \in \ham(M,\omega)\}. 
\end{equation}
An example of our interest will be the cotangent fiber $F_x \subset D_g^*N$ for a closed Riemannian manifold $(N, g)$. By \cite{Che11,Sug16}, the Hofer norm $\|\cdot\|_{\rm{Hofer}}$ on $\ham(M,\omega)$ induces a genuine metric defined as follows, for any $L_1, L_2\in\cL(L)$,
$$\delta_{\rm Hofer}(L_1,L_2)\coloneqq\inf\{\|\phi\|_{\rm{Hofer}}\mid \phi(L_1)=L_2,\,\phi\in\ham(M,\omega)\}.$$
In \cite{Ush14}, for a class of Riemannian manifolds $(N, g)$ including $(S^{n \geq 3}, g_{\rm std})$, it is demonstrated that for any cotangent fiber $F_x$  the metric space $\left(\mathcal L(F_x),d_{\rm{Hofer}}\right)$ contains a rank-$\infty$ quasi-flat.

\medskip

Similarly as above, we consider the spectral norm on  Lagrangian submanifolds in a Liouville domain $W$. Here, we will be interested in those Lagrangian submanifolds that intersect $\partial W$ in a Legendrian way (for details, see Section \ref{admissible-Lag}), for instance, the cotangent fibers $F_x \subset D_g^*N$. To each pair $(\alpha,H)\in H^*(L)\times C^\infty_c([0,1]\times W)$, one associates  Lagrangian spectral invariant $\ell(\alpha,H)$ via wrapped Floer cohomology (see Section \ref{wrapped-Floer}). Recent work \cite{Gon24} shows that in the Liouville domain setting, if compactly supported Hamiltonians $H$ and $F$ satisfy $\phi^1_H(L)=\phi^1_{F}(L)$, then $\ell(\alpha,H)=\ell(\alpha,F)$ for any $\alpha\in H^*(L)$ (cf.~Theorem 1.2 in \cite{Lec08}). Thus $\ell(\alpha,\cdot)$ descends to the orbit space $\cL(L)$. For any $L' = \phi_H^1(L) \in \mathcal L(L)$, denote $\ell(\alpha,L')$ by $\ell(\alpha,H)$ and we define the spectral norm on $\cL(L)$ as follows,
\begin{equation} \label{gamma-L}
\gamma(L')\coloneqq\gamma(L;H)=\ell(\mL,H)+\ell(\mL,\overline{H})
\end{equation}
where $\mL$ is the unit in $H^*(L)$ and $\overline{H}$ generates $\phi^{-1}$. Then we define a bi-invariant metric on $\cL(L)$ as follows, 
$$\delta_{\gamma}(L_1,L_2): =\gamma((\phi_{F}^1)^{-1}\phi_{G}^1(L))$$
for any $L_1=\phi_F^1(L), L_2=\phi_G^1(L)$ (in particular, the non-degeneracy of $\delta_{\gamma}$ is from Theorem 3 in \cite{Gon24}). We emphasize again that $\gamma$ in (\ref{gamma-L}), as well as the induced $\delta_{\gamma}$, is well-defined on $\mathcal L(L)$.

\medskip

The following theorem describes the large scale geometric property of the metric space $\left(\cL(F_x),\delta_\gamma\right)$, where $F_x$ is a cotangent fiber of $D^*_gN$.
\begin{thm}\label{thm-Lag-flat}
    Let $(N,g)$ be a compact connected Riemannian manifold and suppose that there are two non-conjugate points $x_0, x_1\in N$, a homotopy class $\fc\in\pi_0(\mathcal{P}_N(x_0,x_1))$, where $\cP_N(x_0,x_1)$ is the space of smooth paths from $x_0$ to $x_1$ in $N$, and an integer $k$ satisfying all of the following conditions:
    \begin{itemize}
        \item[(i)] There are geodesics from $x_0$ to $x_1$ representing the class $\fc$ and having Morse index $k$, but no geodesics from $x_0$ to $x_1$ representing the class $\fc$ have Morse index $k-1$ or $k+1$.
        \item[(ii)] Only finitely many geodesics from $x_0$ to $x_1$ representing the class $\fc$ have Morse index in $\{k,k+2\}$.
        \item[(iii)] Either $\dim N\ne 2$ or $k\neq 0$.
    \end{itemize}
    Then the metric space $\left(\mathcal{L}(F_{x_0}),\delta_\gamma\right)$  contains a rank-$\infty$ quasi-flat.
\end{thm}
Theorem~\ref{thm-Lag-flat} above enhances the main result in \cite{Gon24} when we  consider the setting $(D^*_gN,\omega_{\rm{can}})$. In this case, \cite{Gon24} only proves that $\left(\cL(F_x),\delta_\gamma\right)$ is unbounded. Moreover, \cite{Gon24} proves  $\left(\cL(L),\delta_\gamma\right)$ is unbounded for any admissible Lagrangian submanifold $L$ in a general Liouville domain $(W,\omega)$ whenever ${\rm \HW}^*(L)\ne 0$. One can post a similar conjecture as in Conjecture~\ref{conj-1} for a relative situation. 

\begin{ex}\label{example}
Here are some examples that the assumption in Theorem~\ref{thm-Lag-flat} holds:\begin{enumerate}
\item (Proposition 3.10 in \cite{Ush14}) Let $(N,g)$ be either a compact semi-simple Lie group with a bi-invariant metric, or a sphere $S^n$ where $n\geq 3$ with its standard metric.  Then the assumption holds for $k=0$.
\item Let $(N,g)$ a sphere $S^1$ with its standard metric.
Then the  assumption  holds with any $k$ since the homotopy class will entirely classify the geodesics connecting two points.
\item Let $(N,g)$ be the closed oriented surface $\Sigma_{g\geq 1}$ with a standard metric. Then the assumption holds for any $k$ as the homotopy class will entirely classify the geodesics connecting two points.
\end{enumerate}
\end{ex}

We emphasize that $S^2$ does {\it not} satisfy the conditions in Theorem~\ref{thm-Lag-flat}, in particular, condition (i). Indeed, for any pair of non-conjugate points $x_0, x_1$ on the standard $S^2$, there exists a geodesic from $x_0$ to $x_1$ with Morse index $k$ for any $k \in \bN$. Compared to the higher-dimensional $S^n$, the issue with $S^2$  is that indices of the generators  corresponding to different geodesics are close to each other, so it is difficult to directly exclude the interplay between generators with adjacent indices (cf.~condition (i) in Theorem \ref{thm-Lag-flat}). Constraints on geodesics from $x_0$ to $x_1$ beyond indices may allow us to estimate the spectral norm furthermore. 

\medskip

The large-scale geometric properties of the absolute case in terms of the Hamiltonian diffeomorphism groups ${\rm Ham}$ and the relative case in terms of the Lagrangian submanifolds $\mathcal L(L)$ are related. Let us give another definition first. 

\begin{dfn}\label{dfn-strong-qf} For any fixed $n \in \bN \cup \{\infty\}$ and a Lagrangian submanifold $L \subset M$, a rank-$n$ quasi-flat $\phi: (\mathbb{R}^n,d_\infty) \to \left(\mathcal L(L), \delta_{\gamma}\right)$ is {\rm in a strong sense} if it is induced by a {\rm homomorphism} $\psi: \mathbb{R}^n \to {\rm Ham}(M, \omega)$ such that $\phi(a) = \psi(a) \cdot L$ for any $a \in \bR^n$. \end{dfn}

Here is a standard example of a homomorphism $\psi: \mathbb{R}^n \to {\rm Ham}(M, \omega)$, which is directly taken from \cite{Ush13}. 

\begin{ex} \label{ex-Usher} Consider any (compactly supported) $H: (M, \omega) \to \bR$ and $f: \bR\ra [0,1]$ with the following properties:
\begin{itemize}
\item ${\rm{supp}}(f)=[\delta,1-\delta]$ for some small real number $\delta>0$.
\item The only local extremum of $f|_{(\delta,1-\delta)}$ is a maximum, at $f(\frac{1}{2})=1$.
\end{itemize}
For $a = (a_1,\cdots, a_n) \in \bR^{n}$ define  
\begin{equation} \label{hom-psi}
\psi(a) \coloneqq \phi^1_{\left(\sum_{i=1}^{n}a_if(2^{i}s-1)\right) \circ H}.
\end{equation}
Obviously, $\psi$ defines a homomorphism since the corresponding Hamiltonian vector field is simply a rescaling of $X_H$. By changing $n$ to $\infty$ in the construction of $\psi$ above, one easily obtains a homomorphism from $\bR^{\infty}$ to ${\rm Ham}(M, \omega)$. For later use (see the proof of Theorem \ref{thm-Lag-flat}), let us denote by $f_a \circ H \coloneqq \left(\sum_i a_if(2^{i}s-1)\right) \circ H$ or for further simplicity, denote $f_a \circ H = : H_a$. 
\end{ex}

In fact, the example above will be repeatedly used (with mild modifications) in the proofs of Theorem~\ref{Ham-infinite-flat} and Theorem~\ref{thm-Lag-flat}.

\begin{thm} \label{abs_rel}
 Let $(N, g)$ be a closed Riemannian manifold, and 
 $L$ be an admissible Lagrangian in $D^*_gN$. If the metric space $\left(\cL(L), \delta_{\gamma}\right)$ contains a rank-$\infty$ quasi-flat $\phi:\bR^\infty\ra \cL(L)$, then there exists a map $\psi:\mathbb{R}^\infty \to \ham(D^*_gN,\omega_{\rm{can}})$, constants $A \geq 1$ and $B \geq 0$ such that 
\begin{equation} \label{abs-rel-leftineq}
 \frac{1}{A}|a-b|_{\infty} - B \leq d_{\gamma}(\psi(a), \psi(b)) 
 \end{equation}
 and for any $a, b \in \bR^{\infty}$, $\phi(a)=\psi(a)\cdot L$. Moreover, if rank-$\infty$ quasi-flat above is in a strong sense (see Definition \ref{dfn-strong-qf}), then for any $n \in \mathbb N$, the metric space  
$\left(\ham(D^*_gN,\omega_{\rm{can}}), d_{\gamma}\right)$ contains a rank-$n$ quasi-flat. 
\end{thm}

We emphasize that the map $\psi$ in Theorem \ref{abs_rel} above is not necessarily a homomorphism.

\medskip

The proof of Theorem~\ref{abs_rel} is essentially based on the $H^*(W)$-module structure of $H^*(L)$, where the action of class $\alpha \in H^*(W)$ on class $\beta \in H^*(L)$ is denoted by $\alpha \cdot \beta$, where up to PSS maps the module structure is given by counting certain pseudo-holomorphic strips with a slit asymptotic to Hamiltonian chords and Hamiltonian closed orbits (for details, see Section \ref{sec-mod}). As a comparison, for closed monotone Lagrangian and the behavior of the spectral invariants under the module structure, see Section 4.3 in \cite{LZ18}.

\begin{theorem}\label{module-structure}
    Let $L$ be an admissible Lagrangian submanifold in the Liouville domain $W$, then for any $\alpha\in H^*(W; \bZ_2)$ and $\beta\in H^*(L; \bZ_2)$, we have $\ell(\alpha\cdot\beta,H\sharp K)\leq\ell(\beta,H)+c(\alpha,K)$ for any Hamiltonian functions $H,K \in C_c^{\infty}([0,1] \times W)$. 
\end{theorem}

Now, Theorem~\ref{abs_rel} quickly follows from Theorem~\ref{module-structure}.

\begin{proof} [Proof of Theorem~\ref{abs_rel}] Take Liouville domain $W=D^*_gN$, classes $\alpha=\mathds{1}_N$, $\beta=\mL$, as well as the Hamiltonian $H=0$, then for any $K \in C_c^{\infty}([0,1] \times W)$, we have $$\ell(\mL,K)\leq c(\mathds{1}_N,K) \,\,\,\,\mbox{and}\,\,\,\, \ell(\mL,\overline{K})\leq c(\mathds{1}_N,\overline{K}).$$ Therefore, by definition (\ref{gamma-phi}) and (\ref{gamma-L}), we have $\gamma(\phi_K^1(L))\leq\gamma(\phi_K^1)$ for any admissible Lagrangian $L$. This implies the desired inequality (\ref{abs-rel-leftineq}), where the constants $A, B$ come from the quasi-flat $\phi: (\mathbb{R}^{\infty}, |-|_{\infty}) \to \left(\cL(L), \delta_{\gamma}\right)$ directly.

Now, for any $n \in \bN$, this quasi-flat $\phi$ is also a rank-$n$ quasi-flat. By hypothesis, if it is in a strong sense, then by Definition \ref{dfn-strong-qf}, $\phi$ is induced by a homomorphism $\psi: \mathbb{R}^n \to {\rm Ham}(W, \omega)$ such that $\phi(a) = \psi(a) \cdot L$. Then we have the following estimations, 
\begin{align*}
d_{\gamma}(\psi(a), \psi(b)) & \leq d_{\rm Hofer}(\psi(a), \psi(b)) & \mbox{(by (\ref{gamma-hofer}))}\\
& = \|\psi(a)^{-1} \psi(b)\|_{\rm Hofer}  & \mbox{(by definition)} \\
& = \|\psi(b-a)\|_{\rm Hofer}. & \mbox{(since $\psi$ is a homomorphism)}
\end{align*}
Denote by $e_1, ..., e_n$ the standard basis of $\bR^n$ and write vector $b-a =  \sum_{i=1}^n x_i e_i$ for some $x_i \in \bR$. Then 
\begin{align*}
\|\psi(b-a)\|_{\rm Hofer} &=\left\|\psi\left(\left(\lfloor|x|_{\infty}\rfloor+1\right)\cdot\sum_{i=1}^n\frac{x_i}{\lfloor|x|_{\infty}\rfloor+1}e_i\right)\right\|_{\rm Hofer}\\
& \leq \left(\lfloor|x|_{\infty}\rfloor+1\right) \cdot \left\|\psi \left(\sum_{i=1}^n\frac{x_i}{\lfloor|x|_{\infty}\rfloor+1}e_i\right)\right\|_{\rm Hofer}\\
& \leq (|x|_{\infty}+1)\cdot \max_{|t|_\infty\leq 1}\|\psi(t)\|_{\rm Hofer}\leq A(n) |b-a|_{\infty} + B(n).
\end{align*}
Here, we take constants $A(n) = B(n) = \max_{|t|_\infty \leq 1} \|\psi(t)\|_{\rm Hofer}$, only depending on $n$. As the domain of $\psi$ is $\bR^n$, both $A(n)$ and $B(n)$ are finite. Together with (\ref{abs-rel-leftineq}), we obtain the desired conclusion. \end{proof}

\begin{remark} \label{rmk-a-b} In practice, when the $\psi$ in the proof of Theorem \ref{abs_rel} is chosen good enough (for instance, the construction in Example \ref{ex-Usher}), constants $A(n)$ and $B(n)$ can be simplified as $A(n) = B(n) = 2$, which independent of $n$. Therefore, the same conclusion as in the second part of Theorem~\ref{abs_rel} also holds for $n = \infty$. \end{remark}

As an immediate corollary, we have the following result. 
\begin{cor} \label{cor-sn}
The metric space $\left(\ham(D^*_gS^{n\neq 2},\omega_{\rm{can}}), d_{\gamma}\right)$ contains a rank-$\infty$ quasi-flat.
\end{cor}

\begin{proof} By Theorem~\ref{thm-Lag-flat},  Example~\ref{example}, Theorem~\ref{abs_rel}, and Remark~\ref{rmk-a-b}. 
\end{proof}

\begin{remark} Note that Theorem~\ref{Ham-infinite-flat} does {\rm not} cover the case where the base manifold $N=S^{n \geq 3}$, while Corollary~\ref{cor-sn}  obtain the expected large scale geometric property for these $S^{n\geq 3}$ indirectly, via Lagrangian cotangent fibers. It would be interesting to find a direct approach for $\ham(D^*_gS^{n\geq 3},\omega_{\rm{can}})$. \end{remark}

\subsection{Method of proofs}\label{sec-method}
The key step in proving Theorem~\ref{Ham-infinite-flat} and Theorem~\ref{thm-Lag-flat} is seeking for a (Floer-theoretic) invariant that serves for the following two purposes:
\begin{itemize}
\item[(a)] this invariant provides a lower bound of spectral norm $\gamma$;
\item[(b)] this invariant detects large-scale geometric properties. 
\end{itemize}
Inspired by works from \cite{Ush13,Ush14} and \cite{KS21}, the ideal candidate is boundary depth $\beta$, derived from the filtration structure of (any) Floer theory that roughly speaking measures the longest time interval that a homological {\it invisible} generator can persist. For more details, especially for $\beta_{\fc}(L_0, L_1; H)$ (boundary depth of wrapped Floer cohomology) and $\beta(H)$ (boundary depth of Hamiltonian Floer cohomology), see Section \ref{wrapped-Floer} and \ref{Ham-Floer}, respectively. In fact, by Corollary 5.4 in \cite{Ush13}, $\beta$ is well-defined on ${\rm Ham}$, but we will not emphasize it here. 

For the purpose (a) above, we have the following result, which verifies the main result Theorem~A in \cite{KS21}, in the setting of a Liouville domain as well as two admissible Lagrangian submanifolds. 
\begin{theorem}\label{betaleqgamma}
  Let $W$ be a Liouville domain and $L_0, L_1\subset W$ be two non-intersecting admissible Lagrangian submanifolds, then we have the following inequalities,
    $$\beta_{\fc}(L_0, L_1;H)\leq\gamma(\phi_H^1(L_0)) \quad \mbox{and} \quad \beta(H)\leq\gamma(\phi_H^1)$$
    for any $H \in C^{\infty}_c([0,1] \times W)$ and $\fc\in\pi_0(\cP(\widehat{L}_0, \widehat{L}_1))$. 
\end{theorem}

For the purpose (b) above, we will modify constructions in both \cite{Ush13} (absolute case) and \cite{Ush14} (relative cases) to obtain large boundary depth $\beta$ which can be comparable with the $|-|_{\infty}$ in $\bR^{\infty}$. More explicitly, for absolute case, we will consider the function $H_a$ constructed in Example \ref{ex-Usher} for any $a \in \bR^{\infty}$ and follow the approach in \cite{Ush13} to show that, for any $\delta>0$, there is a Morse function $G$ with $\|G-H_a\|_{C^0}<\delta$, and 
\begin{equation} \label{large-beta}
\beta_{\rm Morse}(G)\geq-\min_{i\in\bN} a_i
\end{equation}
where $\beta_{\rm Morse}$ is the boundary depth of the filtered Morse cohomology of $G$. 
The condition in Theorem~\ref{Ham-infinite-flat} ensure that $\beta(G)=\beta_{\rm Morse}(G)$, then we have $|\beta(H_a)-\beta_{\rm Morse}(G)|=|\beta(H_a)-\beta(G)|\leq\|H_a-G\|_{\rm Hofer}<2\delta$. Therefore, (\ref{large-beta}) implies that for any $a\in\bR^\infty$, $\beta(H_a)\geq -\min_{i\in\bN} a_i$. 

For relative case, denote by $\beta_{\fc}(F_{x_0},F_{x_1};H)$ the boundary depth of the filtered Lagrangian Floer cohomology defined as in Section 2 in \cite{Ush14} (for two different fibers $F_{x_0}$ and $F_{x_1}$), with a prescribed homotopy class $\fc$ and Hamiltonian $H$. By Proposition 4.3 in \cite{Ush14}, for  any $a\in\bR^\infty$ and a fixed  constant $C>0$, we obtain $\beta_{\fc}(F_{x_0},F_{x_1};H_{a})\geq \|a\|_{\infty}-C$.

\subsection{Summary of current art} So far, the investigation of large-scale geometric properties in either Hamiltonian diffeomorphism group ${\rm Ham}$ or the ${\rm Ham}$-orbit space of a Lagrangian submanifold has made tremendous progresses. In particular, as a more difficult direction (compared with Hofer's metric $d_{\rm Hofer}$) in terms of the spectral norm (metric) $\gamma$ (or $d_{\gamma}$), new phenomena have been discovered. To end this section, let us summarize the current results into a table, where the {\bf bold} parts indicate the results in this paper. 
\vspace{-2mm}
\begin{table}[H]
\centering
\captionsetup{position=bottom}
\begin{tabular}{c|c|c}
                                        & $d_{\rm{Hofer}}$ and $\delta_{\rm Hofer}$  & $d_{\gamma}$ and $\delta_{\gamma}$  \\ \hline
\multirow{3}{*}{absolute}               & \begin{tabular}[c]{@{}l@{}}Some $\ham(M,\omega)$ contain \\ a  rank-$\infty$  quasi-flat (\cite{Ush13,PS23}) \end{tabular} &     \begin{tabular}[c]{@{}l@{}}
   Some $\ham(M,\omega)$ contain \\ a rank-$\infty$  quasi-flat (\cite{KS21})\end{tabular}\\ \cline{2-3} 
   &
   \multirow{2}{*}{\textcolor{black}{\begin{tabular}[c]{@{}l@{}}
     $\ham(D_g^*N,\omega_{\rm can})$ contains \\ 
     a rank-$\infty$ quasi-flat for \\
     some closed manifold $N$\\  
     (Theorem~\ref{Ham-infinite-flat}, \cite{Mil01,Ush14})
\end{tabular}} } &\begin{tabular}[c]{@{}l@{}} $\ham(W,\omega)$ is unbounded\\ for any Liouville domain  $W$\\ with ${\rm SH}^*(W)\ne 0$  (\cite{Mai22})     \\
\end{tabular}         \\ \cline{3-3} 
                                        &
                                      & {~\bf\begin{tabular}[c]{@{}l@{}}
$\bm{\ham(D^*_gN,\omega_{\rm{can}})}$ contains \\ a  rank-$\bm{\infty}$  quasi-flat for \\ some closed  manifold $\bm{N}$ \\ (Theorem~\ref{Ham-infinite-flat})
\end{tabular} }\tikzmark{end}          \\ \hline
\multirow{2}{*}{\begin{tabular}[c]{@{}c@{}}relative \\ (fiber $F_x$)\end{tabular} } & \multirow{2}{*}{\begin{tabular}[c]{@{}l@{}}$\cL(F_q)$ contains a rank-$\infty$ \\ quasi-flat for some closed\\ manifold $N$ (\cite{Ush14})\end{tabular}    } &   \begin{tabular}[c]{@{}l@{}}
  $\cL(F_x)$ is unbounded for \\ any  closed $N$  by \cite{Gon24}    \\
\end{tabular}        \\ \cline{3-3} 
                                        &                   & {\bf\begin{tabular}[c]{@{}l@{}}
$\bm{\cL(F_x)}$ contains a  rank-$\bm{\infty}$ \\ quasi-flat  
for some closed \\  $\bm{N}$ (Theorem~\ref{thm-Lag-flat})\end{tabular} }    \tikzmark{start}          \\ \hline
\begin{tabular}[c]{@{}c@{}}relative\\ (0-section $0_N$) \end{tabular}             &  \begin{tabular}[c]{@{}l@{}}
$\cL(0_N)$ contains a rank-$\infty$\\ quasi-flat for any closed \\ manifold $N$ (\cite{Mil01})
\end{tabular}&  \begin{tabular}[c]{@{}l@{}}  $\cL(0_N)$ is bounded for some \\ closed manifold $N$ (\cite{She22b}) \\ (cf.~Viterbo's conj. \cite{Vit08})  \end{tabular}         \\ 
\end{tabular}
    \caption{Current results of large-scale geometry properties}
\begin{tikzpicture}[overlay, remember picture]
  \draw[->, thick,  bend right] (pic cs:start) to node[pos=0.9, right, xshift=0.7cm, rotate=-90 ] {{\bf Theorem~\ref{abs_rel}}} (pic cs:end);
\end{tikzpicture}
\end{table}

\vspace*{-8mm}

\noindent Note that only the right-down corner, i.e., the spectral metric $\delta_{\gamma}$ on the $ {\rm Ham}$-orbit space of the zero-section $\mathcal L(0_N)$ does {\it not} admit the large-scale geometric property, which is deeply reflected by the celebrated Viterbo's conjecture in \cite{Vit08}.

Let us also mention a technical point for the box in the table above on $d_{\rm Hofer}$ of $\mathcal L(0_N)$: our definition of $\mathcal L(L)$ in (\ref{Lag-orb}) requires that any element in $\mathcal L(L)$ comes as an image of a compactly supported Hamiltonian diffeomorphism on (some fixed) rescaling of $D_g^*N$, where in \cite{Mil01} the quasi-flat is constructed via a family of elements in ${\rm Ham}(T^*N, \omega_{\rm can})$ with arbitrarily large compact supports in $T^*N$.

\subsection*{Acknowledgements} The project was initiated by conversations with Wenmin Gong in summer 2023 when he visited the Institute of Geometry and Physics at USTC, so we are grateful to useful conversations with him. Insightful communications with Michael Usher help to make progresses on this subject. The first author is partially supported by NSFC No.~123B1024. The second author is partially supported by National Key R\&D Program of China No.~2023YFA1010500, NSFC No.~12301081, NSFC No.~12361141812, and USTC Research Funds of the Double First-Class Initiative. We also thank anonymous referee for useful suggestions and corrections.

\section{Various ingredients in Floer theories} \label{sec-2}
\subsection{Hamiltonian Floer theory}\label{Ham-Floer}
In this section, we briefly recall the basic setting of the Hamiltonian Floer theory on  Liouville domains. We refer \cite{BK22} and \cite{Mai22} for more details.

Let $(W^{2n},\omega=d\theta)$ be a Liouville domain. Thus $W$ is a compact manifold with boundary, with $\theta\in\Omega^1(W)$ such that $d\theta$ is symplectic, and there exists a Liouville vector field $V_\theta$ defined on $W$ such that $\iota_{V_\theta}\omega=\theta$. This vector field $V_\theta$ points outside and is transverse to the boundary $\partial W$, which implies that $\alpha\coloneqq\theta|_{\partial W}$ is a contact form on $\partial W$. Denote the flow of $V_\theta$ by $\phi_{V_\theta}^t$. We can extend $W$ to a complete manifold by setting 
\begin{equation} \label{completion}
\widehat{W}\coloneqq W\bigcup (1,\infty)\times\partial W.
\end{equation}
Here, we identify $(1,\infty)\times\partial W$ with $\bigcup_{r \geq 1} \phi^{\log r}_{V_\theta}(\partial W)$. Then the one-form $\theta$ is extended to $\widehat{W}$ by defining $\theta=r\alpha$ for $r\in[1,\infty)$. Consequently,  $(\widehat{W},d\theta)$ becomes a non-compact exact symplectic manifold. 

A Hamiltonian orbit of a smooth Hamiltonian function $H\in C^\infty([0,1]\times\widehat{W})$ is an orbit $x:[0,1]\ra\widehat{W}$ with $\dot{x}(t)=X_H(x(t))$ and $x(0)=x(1)$. The set of contractible $1$-periodic Hamiltonian orbits of $H$ is denoted by $\cP(H)$. An orbit $x$ is said to be non-degenerate if $\det((\phi_H^1)_*|_{T_{x(0)}W}-\mathds{1})\ne 0$. We call a Hamiltonian $H$ non-degenerate if all elements in $\cP(H)$ are non-degenerate. 

For any contractible orbit $x$, we can associate $x$ to a capping $u:\bD\ra\widehat{W}$ such that $u(e^{2\pi i t})=x(t)$, $t\in[0,1]$, where $\bD=\{z\in\bC:|z|\leq 1\}$. We now define a functional $\cA_H:\cL\widehat{W}\ra \bR$ for the pair $(x,u)$, where $\cL\widehat{W}$ denotes the space of contractible orbits in $\widehat{W}$, by
$$\cA_H(x)\coloneqq -\int_{\bD}u^*d\theta+\int_0^1H(x(t))dt,$$
where $u$ is any capping of $x$. By direct computation and Stokes' theorem,
\begin{align*}
    \cA_{H}(x)&=-\int_{\bD}du^*\theta+\int_0^1H(x(t))dt =-\int_0^1x^*\theta+\int_0^1H(x(t)).
\end{align*}
Then $\cA_H(x)$ is independent of the choice of the capping $u$. The set of critical points of $\cA_H$ is just $\cP(H)$.

Recall that the Reeb vector field $R_\alpha$ of $(\partial W,\alpha)$ is determined by
$$
d\alpha(R_\alpha,\cdot)=0\quad\text{ and }\quad \alpha(R_\alpha)=1.
$$
A closed Reeb orbit in $(\partial W,\alpha)$ is a closed orbit $\gamma:[0,T]\ra \partial W$ satisfying $\dot{\gamma}(t)=R_\alpha(\gamma(t))$ and $\gamma(0)=\gamma(T)$. Then the set of periods of all closed  Reeb orbits is denoted by $\spec(\partial W,\alpha)$, which is known to be a closed nowhere dense set in $(0,+\infty)$.  Moreover, if we take $H=h(r)$ on $(r_0,\infty)\times \partial W$ for some $r_0 \geq 1$, then  the Hamiltonian vector field has the form $X_H(r, x)=h'(r)R_\alpha(x)$ on $(r_0,\infty)\times \partial W$. Any Hamiltonian orbit $x$ of $H$ in $(r_0,\infty)\times \partial W$ is restricted to $\{r_1\}\times\partial W$ for some $r_1>r_0$ and corresponds to the Reeb orbit $\gamma(t)=x(t/T)$ with period $T=|h'(r_1)|$.  If the orbit $x$ lies in $(1,\infty)\times\partial W$ and $H=h(r)$ on $(1,\infty)\times \partial W$, then
$$\cA_H(x)=-rh'(r)+h(r),$$
that is, the action of $x$ is equal to the $y$-intercept of the tangent line of the function $y=h(r)$ at $r$. Therefore, we call $H\in C^\infty([0, 1]\times \widehat{W} )$ an {\it admissible} Hamiltonian if $H$ has the form
$$
H(t,r,x)=\mu_Hr+a\quad\text{ on }[1,\infty)\times \partial W,
$$
where $\mu_H\notin  \spec(\partial W,\alpha)$ is a non-negative number (called the slope of $H$). We denote by $\cH$ the set of admissible Hamiltonians. 

Consider a family of smooth $t$-dependent almost complex structures $J_t$ on $\widehat{W}$, where $t$ ranges from 0 to 1. These almost complex structures are compatible with the symplectic form $d\theta$, which means that $J^2=-\mathds{1}$ in ${\rm End}(TW)$ and the pairing $\langle \cdot,\cdot\rangle=d\theta(\cdot,J_t\cdot)$ defines a family of Riemannian metrics on $\widehat{W}$. An almost complex structure $J$ is said to be of contact type on $[r_0,\infty)\times\partial W$ for a given $r_0>0$ if $dr\circ J=-\theta$ for $r\geq r_0$. Let $\mathcal{J}$ denote the set of smooth families $(J_t)_{t\in[0,1]}$ of compatible almost complex structures that are of contact type on $[1,\infty)\times \partial W$. This assumption will simplify our calculation. Now we consider the connecting trajectories to define the differential in the Floer complex. Given $H\in\cH$ and $J\in\cJ$, consider  solutions $u:\bR\times S^1\ra \widehat{W}$ satisfying
$$ \begin{cases}
    \partial_su+J^s_t(\partial_tu-X_{H_t}(u))=0\\
     \lim_{s\ra-\infty}u(s,t)=x\\
     \lim_{s\ra+\infty}u(s,t)=y 
\end{cases}$$
where $x,y$ are Hamiltonian orbits of $H$. Denote the space of the above solutions $u$ by $\widehat{\cM}_{H,J}(x,y)$. There is an $\bR$-translation in the $s$-direction on $\widehat{\cM}_{H,J}(x,y)$, and denote the quotient space by $\cM_{H,J}(x,y)=\widehat{\cM}_{H,J}(x,y)/\bR$. For regular pair $(H,J)$, the space $\widehat{\cM}_{H,J}(x,y)$ is a smooth manifold with dimension
$$\dim\widehat{\cM}_{H,J}(x,y)=\CZ(y)-\CZ(x),$$
where $\CZ(x)$ denotes the Conley-Zehnder index of $x$ (\cite{RS93}). Then the Floer complex is defined by $(\CF^*(H)=\bigoplus_{x\in\cP(H)}\bZ_2\cdot x,d_{H,J})$. The degree of $x\in\cP(H)$ is defined by 
$$|x|=\dfrac{1}{2}\dim\widehat{W}-\CZ(x).$$
The Floer differential $d_{H,J}:\CF^*(H)\ra \CF^{*+1}(H)$ is defined by  
$$d_{H,J}(y)=\sum_{|x|=|y|+1}\sharp_{\bZ_2}\cM_{H,J}(x,y)\cdot x$$
which satisfies $d^2_{H,J}=0$. 

The energy of the connecting trajectory $u$  is defined as 
\begin{align*}
    E_J(u)&=\dfrac{1}{2}\int_0^1\int_0^1|\partial_su|^2_J+|\partial_tu-X_H(u)|^2_Jdsdt\\
    &=\cA_{H}(x)-\cA_{H}(y)\geq 0,
\end{align*}
where $|\cdot|_J$ is the norm with respect to the metric $\omega(\cdot,J\cdot)$. 
We can define a filtration $\ell_{H}$ by 
$$\ell_{H}\left(\sum_{i}a_i x_i\right)\coloneqq\max\{-\cA_{H}(x_i)\mid a_i\ne 0\},$$
where $x_i\in \cP(H)$, then for any $x, y \in {\rm CF}^*(H)$ with $dy = x$, we have $\ell_{H}(y)\geq\ell_{H}(x)$. Then the Floer complex $(\CF^*(H),d_{H,J},\ell_{H})$ is a Floer-type complex (\cite{UZ16}). The Floer cohomology for $H$ is defined to be the quotient space $\ker(d_{H,J})/{\rm im}(d_{H,J})$ which is denoted by $\HF^*(H)$. 

Now we consider the filtered Floer complex. For $a\in\bR\cup\{\pm\infty\}$, we define 
$$\CF^k_{<a}(H)\coloneqq\bigoplus_{\substack{x\in\cP(H)\\|x|=k,\ell_H(x)<a}}\bZ_2\left<x\right>.$$
As $d_{H,J}$ decreases the filtration, the restriction $d_{<a}:\CF^k_{<a}(H)\ra \CF^{k+1}_{<a}(H)$ of the differential is well-defined and $(\CF^*_{<a}(H),d_{<a})$ is a subcomplex of $(\CF^*(H),d_{H,J})$. For $a,b\in\bR\cup\{\pm\infty\}$, we can define the Floer complex in the filtration window $(a,b)$ as the quotient 
$$\CF^{*}_{(a,b)}(H)\coloneqq \CF^*_{<b}(H)/\CF^*_{<a}(H),$$
and we denote the projection of the differential by
$$d_{(a,b)}:\CF^k_{(a,b)}(H)\ra \CF^{k+1}_{(a,b)}(H).$$ 
Then for $a,b,c\in\bR\cup\{\pm\infty\}$, we have the short exact sequence as follows:
$$0 \rightarrow \CF^*_{(a,b)}(H) \xrightarrow{\iota_{(a,b)}^{(a,c)}} \CF^*_{(a,c)}(H)\xrightarrow{\pi_{(a,c)}^{(b,c)}} \CF^*_{(b,c)}(H) \ra 0$$
where $\iota_{(a,b)}^{(a,c)}$ denotes the inclusion and $\pi_{(a,c)}^{(b,c)}$ denotes the projection.

We can define the Floer theory of compactly supported Hamiltonians on Liouville domains $W$ by extending to the linear functions on $\widehat{W}\backslash W$. More precisely, for each $H\in C^\infty_c([0,1]\times W)$, we take a regular $\widehat{H}\in\cH$ such that $\widehat{H}|_W$ is a $C^2$-small perturbation of $H$ and $\mu_H<\min\spec(\partial W,\alpha)$. The filtered Floer cohomology of $H$  on $\widehat{W}$ is defined as $$\HF^*_{(a,b)}(H)\coloneqq\HF^*_{(a,b)}(\widehat{H}).$$
Since we take the slope smaller than the minimum Reeb period to define $\HF^*_{(a,b)}(H)$, the definition above  doesn’t depend on the choice of $\widehat{H}$ (\cite{Vit99}). 
Then the boundary depth of such an $H$ is defined by
$$\beta(H)\coloneqq\sup\left\{a\in\bR\,\left|\, \exists t\in\bR,\, \iota_{(-\infty,t)}^{(-\infty,t+a)}\left(\ker{\iota_{(-\infty,t)}^{(-\infty,+\infty)}}\right)\ne 0\right\}\right.=\beta(\widehat{H}).$$
The spectral invariant for $0\ne\alpha\in H^*(M)$ is defined by $$c(\alpha,H)\coloneqq\inf\{a\in\bR\mid \pi_{(-\infty,+\infty)}^{(a,+\infty)}\circ PSS_H(\alpha)=0 \}=c(\alpha,\widehat{H}),$$
where $PSS_H$ is the PSS map defined in Section \ref{PSS}.
We read the invariants about the regular $H\in \cH$ from the Floer theory. As the regular cases are generic, then any disregarded part can be recovered through $C^0$-continuity.

\subsection{Wrapped Floer cohomology}\label{wrapped-Floer}
To define the wrapped Floer homology, we consider certain well-behaved Lagrangian submanifolds. For more details, see \cite{Rit13,Gon24}.
\begin{dfn}\label{admissible-Lag}
    Let $L^n\subset(W,d\theta)$ be a connected exact Lagrangian submanifold with the Legendrian boundary $\partial L=L\cap\partial W$ such that the Liouville vector field is tangent to $TL$ along the boundary. We call such $L$ an {\rm admissible Lagrangian} if $\theta|_L=dk_L$ for a function $k_L\in C^\infty(L,\bR)$ which vanishes in a neighborhood of the boundary $\partial L$, and the relative Chern class $c_1(W,L)\in H^2(W,L;\bZ)$ satisfies $2c_1(W,L)=0$.
\end{dfn}
Similarly to (\ref{completion}), one can extend admissible $L$ to an exact non-compact Lagrangian in $\widehat{W}$ by setting
$$\widehat{L}\coloneqq L\bigcup\phi^t_{V_\theta}(\partial L)$$
and setting $k_L=0$ on $\widehat{L}\backslash L$.
\begin{ex}\label{disk-conormal}
    Let $D_g^*N$ be the disk cotangent bundle of a closed manifold $(N,g)$. The disk conormal bundle
    $$\nu^*K=\left\{(q,p)\in D_g^*N|_K \, | \, p(v)=0, \forall v\in T_qK \right\}$$
    of a closed submanifold $K\subset N$ is an admissible Lagrangian with $k_L=0$. If we take $K=\{x\}$, then $F_x$ is an admissible Lagrangian (called the fiber at $x$) with $k_{F_x}=0$.
\end{ex}

Fixing two admissible Lagrangians $L_0, L_1\subset W$, a Reeb chord of period $T$ is a map $\gamma:[0,T]\to\partial W$ satisfying
$$
\dot{\gamma}(t)=R(\gamma(t))\quad\hbox{and}\quad \gamma(0)\in\partial L_0,\, \gamma(T)\in\partial L_1.
$$
Then the set of periods of Reeb chords is denoted by $\spec(\partial L_0,\partial L_1; \theta)$, which is known to be a closed nowhere dense subset in $(0,+\infty)$. 
A Hamiltonian chord of a smooth Hamiltonian function $H\in C^\infty([0,1]\times \widehat{W})$ is a chord with properties
$$
\dot{x}(t)=X_H(x(t)) \quad\hbox{and}\quad x(0)\in\widehat{L}_0,\, x(1)\in \widehat{L}_1,
$$
where $X_H$ is the Hamiltonian vector field given by $d\theta(X_H,\cdot)=-dH_t$ with $H_t\coloneqq H(t,\cdot)$. Clearly,  Hamiltonian chords of $H$ correspond to the intersection points in $\phi^1_H(\widehat{L}_0)\cap \widehat{L}_1$ where $\phi_H^1$ is the time one map of the flow of $X_H$. 

Moreover, for Hamiltonian $H$ with $H=h(r)$ on $(r_0,\infty)\times \partial W$, any Hamiltonian chord $x$ of $H$ in this region is restricted to $\{r_1\}\times\partial W$ for some $r_1>r_0$ with $H(x(t))=h(r_1)$, and corresponds to the Reeb chord $\gamma(t)=x(t/T)$ with period $T=|h'(r_1)|$. Therefore, we call $H\in C^\infty([0,1]\times \widehat{W})$ an {\it admissible Hamiltonian} if $H$ has the form
\begin{equation} \label{extension-H}
H(t,r,x)=\mu_Hr+a\quad\text{ on }[1,\infty)\times \partial W.
\end{equation}
Here, $\mu_H\notin  \spec(\partial L_0, \partial L_1,\theta)$ is a non-negative scalar known as the slope of $H$.  We denote by $\cH$ the set of all admissible Hamiltonians.  

Now we follow \cite{Ush14} to define the Floer complex $(\CW^*_\fc(L_0,L_1;H),\partial)$ for  Liouville domain $W$ and admissible Lagrangians $L_0, L_1$, where $\fc\in\pi_0(\mathcal{P}(\widehat{L}_0,\widehat{L}_1))$ and $H\in C^\infty_c([0,1]\times \widehat{W})$. For each $\fc\in\pi_0(\mathcal{P}(\widehat{L}_0,\widehat{L}_1))$, fix a path $\gamma_\fc:[0,1]\ra \widehat{W}$ which represents the class $\fc$. Denote $\mathcal{P}_\fc(\widehat{L}_0,\widehat{L}_1)$ as the space of paths  representing the class $\fc$ from $\widehat{L}_0$ to $\widehat{L}_1$. The generator of $\CW^*_\fc(L_0,L_1;H)$ is the critical points of the function $\cA_{L_0,L_1;H}:\mathcal{P}_\fc(\widehat{L}_0,\widehat{L}_1)\ra\bR$ defined by 
\[\cA_{L_0,L_1;H}(\gamma)\coloneqq -\int_{[0,1]^2}u^*d\theta+\int_0^1H(t,\gamma(t))dt-\int_{0}^{1}\gamma_{\fc}^*\theta
\]
where $u:[0,1]^2\ra \widehat{W}$ is a capping of $\gamma$ with $u(0,\cdot)=\gamma_\fc(\cdot), u(1,\cdot)=\gamma(\cdot), u(\cdot, 0)\in\widehat{L}_0$ and $u(\cdot, 1)\in\widehat{L}_1.$ This is slightly different from the usual definition as in \cite{Ush14}. We add the last term to remove the dependence of the choice of $\gamma_{\fc}$, which is helpful to estimate the energy of the connecting strip in Section~\ref{section-product}. Each critical point $\gamma$ of $\cA_{L_0,L_1;H}$ satisfies $\gamma'(t)=X_H(t,\gamma(t))$, and they are in bijection with the intersection points $\gamma(0)\in\widehat{L}_0\cap(\phi_H^1)^{-1}(\widehat{L}_1)$. Note that we define the functional $\cA_{L_0,L_1;H}(\gamma)=\cA_{L_0,L_1;H}((\gamma,u))$, we need to show that $\cA_{L_0,L_1;H}(\gamma)$ is independent of the choice of $u$. By direct computations and Stokes' theorem,
\begin{align*}
    \cA_{L_0,L_1;H}(\gamma)=&-\int du^*\theta+\int_0^1H(\gamma(t))dt-\int_{0}^{1}\gamma_{\fc}^*\theta\\
    =&-\int_0^{1}\gamma^*\theta+\int_0^1u(1,\cdot)^*\theta-\int_0^1u(0,\cdot)^*\theta+\int_0^1H(\gamma(t))dt\\
    =&-\int_0^{1}\gamma^*\theta+\int_0^1u(1,\cdot)^*dk_{L_1}-\int_0^1u(0,\cdot)^*dk_{L_0}+\int_0^1H(\gamma(t))dt\\
    =&-\int_0^1 \gamma^*\theta+\int^1_0 H(\gamma(t))dt+k_{L_1}\big(\gamma(1)\big)-k_{L_0}\big(\gamma(0)\big)\\
    &-k_{L_1}\big(\gamma_{\fc}(1)\big)+k_{L_0}\big(\gamma_{\fc}(0)\big).
\end{align*}
where $\theta=dk_{L_i}$ on $\widehat{L_i}$, $i=0, 1$. If we take the Liouville domain $W=D_g^*N$ for a closed manifold $N$ and $L_0, L_1$ the disk conormal bundle in Example \ref{disk-conormal}, then $k_{L_i}=0$, and $\cA_{L_0,L_1;H}(\gamma)=-\int_0^1\gamma^*\theta+\int_0^1H(\gamma(t))dt.$
The set of critical points of $\cA_{L_0,L_1;H}$ is denoted by ${\rm Crit}_{\fc}(L_0,L_1;H)$ for each homotopy class $\fc$.

A chord $x$ is said to be non-degenerate if the vector spaces $T_{x(1)}\widehat{L}_1$ and $d\phi_H^1(T_{x(0)}\widehat{L}_0)$ are transverse. We call an admissible Hamiltonian $H\in\cH$ non-degenerate with respect to $\widehat{L}_0$ and $\widehat{L}_1$ if all chords $x$ between $\widehat{L}_0$ and $\widehat{L}_1$ are non-degenerate. 

The wrapped Floer complex is defined by 
\begin{equation}\label{eq-chain-complex}
\CW^*_{\fc}(L_0,L_1;H)\coloneqq\bigoplus_{x\in{\rm Crit}_{\fc}(L_0,L_1;H)}\bZ_2 \left<x\right>.
\end{equation}
To define the differential, we need to consider the connection trajectories.   
Given $J\in\cJ$ and $H\in\cH^{reg}$, consider the connecting trajectory $u:\bR\times[0,1]\ra \widehat{W}$ satisfy
$$\begin{cases}
    \partial_su+J^s_t(\partial_tu-X_{H_t}(u))=0,\\
        u(\bR,0)\subset\widehat{L}_0, \quad u(\bR,1)\subset\phi_H^1(\widehat{L}_1),\\
     \lim_{s\ra-\infty}u(s,t)=x,\quad\lim_{s\ra+\infty}u(s,t)=y.   
\end{cases}$$
Denote the space of the above solutions $u$ of the finite energy by $\widehat{\cM}_{H,J}(x,y)$. There is a $\bR$-translation in the $s$-direction on $\widehat{\cM}_{H,J}(x,y)$, and the quotient space $\cM_{H,J}(x,y)=\widehat{\cM}_{H,J}(x,y)/\bR$. For regular pair $(H,J)$, the space $\widehat{\cM}_{H,J}(x,y)$ is a smooth manifold with dimension
$$\dim\widehat{\cM}_{H,J}(x,y)=\mu(x)-\mu(y),$$
where $\mu$ is the index of the chord.

Then the Floer differential is defined by  
$$d_{H,J}:\CW^*_{\fc}(L_0,L_1;H)\ra \CW^*_{\fc}(L_0,L_1;H)$$
$$d_{H,J}(y)=\sum_{\mu(x)=\mu(y)+1}\sharp_{\bZ_2}\cM_{H,J}(x,y)\cdot x.$$
This map has $d^2_{H,J}=0$, and therefore $\CW^*_{\fc}(L_0,L_1;H)$ is a chain complex over $\bZ_2$. The energy of the connect trajectory $u$ is defined as 
$$E_J(u)\coloneqq\dfrac{1}{2}\int_0^1\int_0^1|\partial_su|^2_J+|\partial_tu-X_H(u)|^2_Jdsdt$$
where $|\cdot|_J$ is the norm with respect to the metric $\omega(\cdot,J\cdot)$. One verifies that $E(u)=\cA_{L_0,L_1;H}(x)-\cA_{L_0,L_1;H}(y)\geq 0.$ We can define a filtration $\ell_{L_0,L_1;H}$ by 
\begin{equation}\label{eq-fil}
\ell_{L_0,L_1;H}\left(\sum_{i}a_i x_i\right)\coloneqq\max\{-\cA_{L_0,L_1;H}(x_i)\mid a_i\ne 0\}
\end{equation}
where $x_i\in {\rm Crit}_{\fc}(L_0,L_1;H)$, 
then for any $x,y\in\CW^*_{\fc}(L_0,L_1;H)$ with $dy=x$, we have $\ell_{L_0,L_1;H}(y)\geq\ell_{L_0,L_1;H}(x)$. The wrapped Floer complex $(\CW^*_{\fc}(L_0,L_1;H),d_{H,J},\ell_{L_0,L_1;H})$ is a Floer-type complex. The wrapped Floer cohomology for $(L_0,L_1;H)$ is defined to be the quotient space $\ker(d_{H,J})/{\rm im}(d_{H,J})$ which is denoted by $\HW^*_{\fc}(L_0,L_1;H)$. 

For $a\in\bR\cup\{\pm\infty\}$, we define the filtered wrapped Floer complex
$$\CW^*_{\fc, <a}(L_0,L_1;H)\coloneqq\{x\in \CW^*_{\fc}(L_0,L_1;H)\, | \,\ell_{L_0,L_1;H}(x)<a \}.$$
Since the filtration $\ell_{L_0,L_1;H}$ does not increase along the  gradient flows, the differential $d_{H,J}$ decreases the filtration $\ell_{L_0,L_1;H}$, and hence the vector space $\CW^*_{\fc,<a}(L_0,L_1;H)$ is a subcomplex of the wrapped Floer complex $\CW^*_{\fc}(L_0,L_1;H)$. For all $a\in\bR$, the filtered wrapped Floer cohomology of $\CW^*_{\fc,<a}(L_0,L_1;H)$ is denoted ${\rm HW}^*_{\fc,<a}(L_0,L_1;H)$.

For $k \in \bZ$ and $a \in \bR \cup \{\pm \infty\}$, we define 
$$\CW^k_{\fc,<a}(L_0,L_1;H)\coloneqq\bigoplus_{\substack{x\in{\rm Crit}_{\fc}(L_0,L_1;H)\\|x|=k,\ell_{L_0,L_1;H}(x)<a}}\bZ_2\left<x\right>.$$
As $d_{H,J}$ decreases the filtration, $d_{<a}:\CW^k_{\fc,<a}(L_0,L_1;H)\ra \CW^{k+1}_{\fc,<a}(L_0,L_1;H)$ of the differential is well-defined and thus $(\CW^*_{\fc,<a}(L_0,L_1;H),d_{<a})$ is a subcomplex of $(\CW^*_{\fc}(L_0,L_1;H),d_{H,J})$. For $a,b\in\bR\cup\{\pm\infty\}$, we can define the Floer complex in the filtration window $(a,b)$ as the quotient 
$$\CW^{*}_{\fc,(a,b)}(L_0,L_1;H)\coloneqq \CW^*_{\fc,<b}(L_0,L_1;H)/\CW^*_{\fc,<a}(L_0,L_1;H),$$
and we denote the projection of the differential by
$$d_{(a,b)}:\CW^k_{\fc,(a,b)}(L_0,L_1;H)\ra \CW^{k+1}_{\fc,(a,b)}(L_0,L_1;H).$$ 
Then for $a,b,c\in\bR\cup\{\pm\infty\}$, we have the short exact sequence as follows:
$$0 \ra\CW^*_{\fc,(a,b)}(L_0,L_1;H)\xrightarrow{\iota_{(a,b)}^{(a,c)}}\CW^*_{\fc,(a,c)}(L_0,L_1;H) \xrightarrow{\pi_{(a,c)}^{(b,c)}}\CW^*_{\fc,(b,c)}(L_0,L_1;H) \ra 0$$
where $\iota_{(a,b)}^{(a,c)}$ denotes the inclusion and $\pi_{(a,c)}^{(b,c)}$ denotes the projection.
We can define the boundary depth   
\begin{equation} \label{dfn-beta-two-Lag}
\beta_{\fc}(L_0,L_1;H)\coloneqq\sup_{x\in {\rm{Im}}d_{H,J}\backslash\{0\}}\inf\{\ell(y)-\ell(x)\mid d_{H,J}(y)=x\}.
\end{equation}

When $L_ 0= L_1$, the construction above recovers the setting in \cite{Gon24}. In this case, we will only consider the special case where the homotopy class $\fc=[\rm pt]$. To simplify the notation, we denote 
\begin{equation} \label{red-CW}
\CW^*_{(a,b)}(L; H)\coloneqq\CW^*_{[\rm pt],(a,b)}(L,L;H).
\end{equation}

We can define the Floer theory of compactly supported Hamiltonians on Liouville domains $W$ by extending to be in the form of linear functions on $\widehat{W}\backslash W$. More precisely, for each $H\in C^\infty_c([0,1]\times W)$, we take a regular $\widehat{H}\in\cH$ such that $\widehat{H}|_W$ is a $C^2$-small perturbation of $H$ and $\mu_H$ in (\ref{extension-H}) satisfies $0<\mu_H<\min\spec(\partial L,\partial L,\theta)$. The filtered Floer cohomology of $H$ is defined by 
\begin{equation} \label{FL-LD}
\HW^*_{(a,b)}(L; H)\coloneqq\HW^*_{(a,b)}(L; \widehat{H})
\end{equation}
where $\HW^*_{(a,b)}(L; \widehat{H})$ is defined via the cochain complex $\CW^*_{(a,b)}(L; \widehat{H})$ as above. As long as $\mu_H$ satisfies the condition $0<\mu_H<\min\spec(\partial L,\partial L,\theta)$ as above, the definition $\HW^*_{(a,b)}(L,H)$ in (\ref{FL-LD}) doesn't depend on the choice of $\widehat{H}$. 
Then the boundary depth of $H$ is $\beta(L; H)\coloneqq\beta_{[\rm pt]}(L,L;\widehat{H})$. Similarly, the spectral invariant for $0\ne\alpha\in H^*(L; \bZ_2)$ is defined by $$\ell(\alpha,H): =\inf\{a\in\bR\mid \pi_{(-\infty,\infty)}^{(a,\infty)}\circ PSS_H(\alpha)=0 \},$$
where $PSS_H$ is the PSS map defined in Section \ref{PSS}.

\subsection{Product structure of wrapped Floer cohomology}\label{section-product}
Now we recall the product structure of wrapped Floer cohomology.  Consider a $2$-dimensional disk $\bD$ with three boundary points $z_i\in\partial\bD$, $i=0,1,2$ removed. We denote the arc along $\partial\bD$ connecting $z_i$ to $z_{i+1}$ by $C_i$, with the convention that $z_3\coloneqq z_0$. Let $j$ be a complex structure on $\bD$.  Near every boundary punctures we equip a strip-like end which can be biholomorphically mapped onto the semi-infinite strips 
$$Z_\pm=\mathbb{R}_\pm\times [0,1]$$
with the coordinate $(s, t)$ and the standard complex structure, i.e. $j\partial_s=\partial_t$. More precisely, for $i=1,2$ we consider a positive strip-like end near $z_i$ which is a holomorphic embedding $\kappa_i:\mathbb{R}_+\times[0,1]\ra\bD$ satisfying 
$$\kappa_i^{-1}(\partial\bD)=\mathbb{R}_+\times\{0,1\} \text{ and }\lim_{s\ra+\infty}\kappa_i(s,\cdot)=z_i.$$
Near $z_0$  we equip a negative strip-like end in a similar way. Moreover, we require that these strip-like ends (the images of $\kappa_i$) are pairwise disjoint. We can described $\bD$ as a strip with a slit: consider the disjoint union
$\bR\times[0,0.5]\sqcup\bR\times[0.5,1] $
and identify $(s, 0.5^-)$ with $(s,0.5^+)$ for every $s\leq 0$ (see Figure \ref{fig-1}). 

\begin{figure}[ht]
	\centering
\includegraphics[scale=0.8]{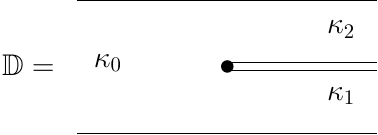}
\caption{A strip with a slit. }\label{fig-1}
\end{figure}

Giving regular pairs $(H_i,J_i)$, $i=0,1,2$, we want to define the product
$$\ast: \CW^*_{\fc_0}(L_0, L_1; H_1)\otimes \CW^*_{\fc_1}(L_1, L_2; H_2)\ra \CW^*_{\fc_0\ast \fc_1}(L_0, L_2; H_0).$$
where $\fc_0\in\pi_0(\cP(\widehat{L}_0,\widehat{L}_1))$, $\fc_1\in\pi_0(\cP(\widehat{L}_1, \widehat{L}_2))$ and $\fc_0\ast\fc_1\in\pi_0(\cP(\widehat{L}_0, \widehat{L}_2))$. Here we use $\fc_0\ast\fc_1$ denote the composition of the homotopy class. If $\fc_0=[\gamma_0]$ and $\fc_1=[\gamma_1]$, then $\fc_0\ast\fc_1=[\gamma_0\ast\gamma_1]$, where
\[
(\gamma_0\ast\gamma_1)(t)=\begin{cases}
\gamma_0(2t), & 0\leq t\leq\frac{1}{2}\\
\gamma_1(2t-1), & \frac{1}{2}\leq t\leq 1
\end{cases}.
\]
To this end, we need to choose compatible perturbation data. Denote by $K\in\Omega^1(\bD,C^{\infty}(\widehat{W},\bR))$ a Hamiltonian $1$-form, $X_K\in\Omega^1(\bD,\Gamma^\infty(T\widehat{W}))$ a Hamiltonian vector field associated with $K$. Let $J=\{J_z\}_{z\in\bD}$ be a family of almost complex structures parameterized by $\bD$ that satisfies $J_{\kappa_i(s,t)}=(J_i)_t$, $i=0,1,2$. Additionally, we require that $K|_{TC_i}\equiv 0$ and $\kappa_i^*K=(H_i)_t\otimes dt$, $i=0,1,2$. Next, consider the inhomogeneous $\overline{\partial}$-equation
\begin{equation}\label{inhom}
    \begin{cases}
    (du-X_K)^{0,1}=0\\
    u:\bD\ra\widehat{W},\quad     u(C_i)\subset\widehat{L}_i, \quad i=0,1,2,\\
    \displaystyle\lim_{s\ra\pm\infty}u(\kappa_i(s,\cdot))=z_i, \quad i=0,1,2,
\end{cases}
\end{equation}
where each $z_i$ is a chord of the Hamiltonian $H_i$ with respect to $\widehat{L}_i$ and $\widehat{L}_{i+1}$, with the convention that $\widehat{L}_3\coloneqq \widehat{L}_0$. Then we can define a chain map 
\begin{align*}
    \ast: \CW^*_{\fc_0}(L_0, L_1; H_1)\otimes \CW^*_{\fc_1}(L_1, L_2; H_2)\ra \CW^*_{\fc_0\ast\fc_1}(L_0, L_2; H_0)\\
    z_1\ast z_2=\sum_{\mu(z_0)=\mu(z_1)+\mu(z_2)} \sharp_{\bZ_2}\cM(K,J;z_0,z_1,z_2)z_0
\end{align*}
where the moduli space $\cM(K,J;z_0,z_1,z_2)$ is the space of the solutions $u$ to \eqref{inhom} with finite energy. This implies a well-defined bilinear map on homology
$$\ast:\HW^*_{\fc_0}(L_0, L_1; H_1)\otimes \HW^*_{\fc_1}(L_1, L_2; H_2)\ra \HW^*_{\fc_0\ast \fc_1}(L_0 ,L_2; H_0).$$
Meanwhile, the energy of such solution $u$ to \eqref{inhom} is defined by 
$$E(u)\coloneqq\dfrac{1}{2}\int_{\cD}\|du-X_K\|^2{\rm{vol}}_\cD.$$
For any Hamiltonian $1$-form $K\in\Omega^1(\bD,C^\infty(\widehat{W},\bR))$, we denote the curvature term of the Hamiltonian $1$-form $K$ as 
$$R(K)\coloneqq dK-\{K,K\}\in\Omega^2(\bD,C^\infty(\widehat{W},\bR)).$$ 
Then in local coordinates $(s,t)$ on $\Sigma\subset\bD$, we have 
$$R(K)(\partial_s,\partial_t)=\partial_s K(\partial_t)-\partial_t K(\partial_s)-\{K(\partial_s),K(\partial_t)\}$$
where $\{\cdot,\cdot\}$ denotes the Poisson bracket.  We have the following no escape lemma to control the solution to \eqref{inhom} (\cite{Rit13})
\begin{lemma}\label{lem:bound}
If $R(K)(\partial_s,\partial_t)\leq 0$, then $u(\mathbb{D})$ lies in $W$ for all solutions $u$ to \eqref{inhom}.
\end{lemma}

To obtain the desired $K$ (in order to satisfy the condition in Lemma \ref{lem:bound}), the construction of $K$ is based on the following lemma.
\begin{lemma}
Let $\{H^s\}_{s\in[0,1]}$ be a path of admissible Hamiltonians with $\phi^1_{H^s}=\phi^1_{H}$ for all $s\in[0,1]$. Then there exists $K\in\Omega^1([0,1]^2,C^{\infty}(\widehat{W},\bR))$ such that $K|_{T([0,1]\times\{i\})}\equiv 0$, $K|_{\{i\}\times[0,1]}=H_i\otimes dt$, $i=0,1$ and $R(K)=dK-\{K,K\}\equiv 0$ on $[0,1]^2$.
\end{lemma}
\begin{proof}
Let $\phi(s,t)=\phi_{H^s}^t$ be a two-parameter family of Hamiltonian diffeomorphisms, where $H^s$ is the Hamiltonian generates the vector field $\pdv{\phi}{t}\circ\phi^{-1}.$ Denote $G$ as the family of functions $G=G(s,t,x)$ where $G^s=G(s,\cdot,\cdot):[0,1]\times\widehat{W}\ra\bR$ is the Hamiltonian generates the vector field $\pdv{\phi}{s}\circ\phi^{-1}.$ Banyaga's Lemma (\cite{Ban97}, Proposition 3.1.5) implies that $$\pdv{X_H}{s}-\pdv{X_G}{t}+[X_G,X_H]=0.$$
Taking the interior product with $\omega$ leads to the conclusion that
$$\pdv{H}{s}-\pdv{G}{t}-\{G,H\}=c(s,t)$$
for some  $c$ only depending on $(s,t)$. Since all terms vanish on $\widehat{W}\backslash W$, it follows that  $c(s,t)\equiv 0$.  The the Hamiltonian $1$-form $K\coloneqq Hdt+Gds$ satisfies
$$R(K)(\partial_s,\partial_t)=\pdv{H}{s}-\pdv{G}{t}-\{G,H\}=0.$$
Since $\phi(s,0)= \mathds{1}$ and $\phi(s,1)=\phi^1_{H^0}$, it follows that $G|_{[0,1]\times\{0,1\}}\equiv 0$ and $K|_{T([0,1]\times\{i\})}\equiv 0$ for $i=0,1$. By reparameterizing $H^s$ such that $\partial_sH^s=0$ near $s=0,1$, we have  $G|_{([0,\varepsilon]\cup[1-\varepsilon])\times[0,1]}\equiv 0$ for a small $\varepsilon>0$ and $K|_{\{i\}\times[0,1]}=H_i\otimes dt$ for $i=0,1$.
\end{proof}
Let $K_0$ be a Hamiltonian $1$-form connecting $(H_1\sharp H_2)\otimes dt$ and $(H_1'\ast H_2') \otimes dt$, where $H_1'\ast H_2'$ is defined as follows,
$$(H_1'\ast H_2')_t(x)=\begin{cases}
(H_1')_{2t}(x) &\mbox{if} \,\,\,\,  0\leq t\leq\dfrac{1}{2}, x\in W\\
(H_2')_{2t-1}(x)\,\,\,\,&\mbox{if} \,\,\,\,\dfrac{1}{2}\leq t\leq 1, x\in W\\
2(H_1)_{2t}(x) &\text{if} \,\,\,\, x\in\widehat{W}\backslash W
\end{cases}$$
and $H'=\chi'(t) H$ for  a suitable smooth monotone surjection $\chi:[0,1/2]\ra[0,1]$. Let $K_i$ be the Hamiltonian $1$-form connecting $H'_i \otimes dt$ and $2H_i \otimes dt$ for $i=1,2$. Consider a Hamiltonian $1$-form $K$ is defined as follows, 
$$K=
\begin{cases}
(H_1\sharp H_2)\otimes dt & \text{if} (s,t)\in(-\infty,-1)\times(0,1)\\
K_0(s+1,t) &\text{if} (s,t)\in [-1,0]\times(0,1)\\
K_1(s,2t) &\text{if} (s,t)\in(0,1]\times(0,0.5)\\
K_2(s,2t-1)&\text{if} (s,t)\in(0,1]\times (0.5,1)\\
2H_1\otimes dt &\text{if} (s,t)\in(1,\infty)\times(0,0.5)\\
2H_2\otimes dt&\text{if} (s,t)\in(1,\infty)\times(0.5,1)
\end{cases}
$$
then $R(K)\equiv 0$,  $K|_{T\partial\bD}\equiv 0$, $\kappa_i^*K=(H_i)_t\otimes dt$ for $i=0,1,2$. Note that $\kappa_i^*(2dt)=dt$.

\begin{figure}[ht]
	\centering
\includegraphics[scale=1]{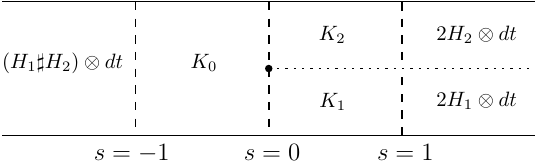}
\caption{The desired Hamiltonian $1$-form $K$.}\label{fig-Ham-form}
\end{figure}

Applying the local result $\dfrac{1}{2}\|du-X_K\|^2=u^*\omega-d(K\circ u)+R\circ u$ (see \cite{MS12}, Section 8.1), we get
\begin{align*}
    0\leq E(u)=&\int_\bD\dfrac{1}{2}\|du-X_K\|^2\rm{vol}_\bD\\
    =&\int_\bD d(u^*\theta-K\circ u)\\
    =&\int_{[0,1]} -z_0^*\theta+ z_1^*\theta+z_2^*\theta-\int_{[0,1]}(H_1+H_2-H_1\sharp H_2)dt\\
    =&\cA_{L_0, L_2; H_1\sharp H_2}(z_0)-\cA_{L_0, L_1; H_1}(z_1)-\cA_{L_1, L_2; H_2}(z_2)\\
    =&-\ell_{L_0, L_2; H_1\sharp H_2}(z_0)+\ell_{L_0, L_1; H_1}(z_1)+\ell_{L_1, L_2; H_2}(z_2).
\end{align*}
In this way, the product $\ast$ realizes a product between filtered chain complexes, 
$$\ast:\CW^*_{\fc_0,<a}(L_0, L_1; H_1)\otimes \CW^*_{\fc_1,<b}(L_1, L_2; H_2)\ra \CW^*_{\fc_0\ast\fc_1, <a+b}(L_0, L_2; H_1\sharp H_2)$$
for any $a, b\in\bR$. In particular, we have the filtered map for any $x\in \CW^k_{\fc_0}(L_0, L_1; H_1)$ with $\ell_{L_0, L_1; H_1}(x)=a$,
\begin{equation} \label{x-product}
x\ast-:\CW^*_{\fc_1}(L_1, L_2; H_2)\ra \CW^{*+k}_{\fc_0\ast\fc_1}(L_0, L_2; H_1\sharp H_2)[a],
\end{equation}
where the right-hand side $(\CW^{*+k}_{\fc_0\ast\fc_1}(L_0, L_2; H_1\sharp H_2)[a])_{<t}\coloneqq\CW^{*+k}_{\fc_0\ast\fc_1,<t+a}(L_0, L_2; H_1\sharp H_2)$ for any $t\in\bR$.

\medskip

Similarly, we can define a higher order product $\mu_3$, with the Floer perturbation data based on the moduli space of discs with $4$ punctured points on boundary. Denote this moduli space by $\cM_4$ and there exists
\begin{align*}
\mu_3: \CW^*_{\fc_0, <a}(L_0, L_1; H_1)\otimes \CW^*_{\fc_1,<b}(L_1, L_2; H_2)\otimes &\CW^*_{\fc_2,<c}(L_2, L_3; H_3)\ra \\
&\CW^*_{\fc_0\ast\fc_1\ast\fc_2,<a+b+c+\varepsilon}(L_0, L_3; H_1\sharp H_2\sharp H_3)
\end{align*}
for any small $\varepsilon>0$ (see \cite{KS21}, Section 2.5), $\fc_0\in\pi_0(\cP(\widehat{L}_0,\widehat{L}_1))$, $\fc_1\in\pi_0(\cP(\widehat{L}_1, \widehat{L}_2))$,  $\fc_2\in\pi_0(\cP(\widehat{L}_2, \widehat{L}_3))$, and $\fc_0\ast\fc_1\ast\fc_2\in\pi_0(\cP(\widehat{L}_0, \widehat{L}_3))$. 
From the consideration of the compactness of $\cM_4$ (see Figure \ref{fig-comp-mu3}), we have a relation that mixes $d, \ast$ and $\mu_3$, 
\begin{equation} \label{mu3-rel}
d(\mu_3(x,y,z))+\mu_3(dx,y,z)+\mu_3(x,dy,z)+\mu_3(x,y,dz)=(x\ast y)\ast z+x\ast(y\ast z)
\end{equation}
where $x\in \CW^*_{\fc_0,<a}(L_0, L_1; H_1), y\in \CW^*_{\fc_1,<b}(L_1, L_2; H_2), z\in \CW^*_{\fc_2,<c}(L_2, L_3; H_3)$ and $(x\ast y)\ast z,x\ast(y\ast z)\in \CW^*_{\fc_0\ast\fc_1\ast\fc_2,<a+b+c+\varepsilon}(L_0, L_3; H_1\sharp H_2\sharp H_3)$ composed with suitable inclusion.
\begin{figure}[ht]
    \centering
    \includegraphics[scale=1.4]{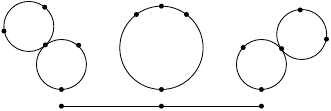} 
    \caption{Compactification of $\cM_4$.}\label{fig-comp-mu3}
\end{figure}

\subsection{PSS morphism}\label{PSS}
We introduce the PSS morphism in absolute and relative cases. For more details, see \cite{Lec08,Rit13}.
Fix a Riemannian metric $g$ on $\widehat{W}$. 
In order to control the energy (while working on non-compact $\widehat{W}$), we will concentrate on the functions that adapted to $W$, which means that the following conditions hold:
\begin{enumerate}
    \item No critical points of $f$ occur in $\widehat{W}\backslash W\cong\partial W\times(1,\infty)$;
    \item The gradient vector field $\nabla f$ of $f$ with respect to $g$ points outward along $\partial W$;
    \item  $f|_L$ is a $C^2$-smal Morse function and $(f,g)$ is a Morse-Smale pair.
\end{enumerate}
We denote $\text{Crit}_k(f)$ the set of critical points of $f$ witht Morse index $k$. Then we  define the Morse complex ${\rm CM}^k(f,g)=\bZ_2\left<\text{Crit}_k(f)\right>$ and a differential $d$ by counting isolated negative gradient flow lines of $\nabla f$, i.e.
$$d: {\rm CM}^k(f,g)\ra {\rm CM}^{k+1}(f,g),$$
$$d q=\sum_{{\rm{Ind}}_f(p)=k+1}\sharp_{\bZ_2}\cM_{p,q}(f,g)\cdot p,$$
where the moduli space $\cM_{p,q}(f,g)$ is given by
$$\cM_{p,q}(f,g)\coloneqq\{\gamma\in C^\infty(\bR,L)\, | \,\dot{\gamma}=-\nabla f(\gamma(t)), \gamma(-\infty)=p, \gamma(\infty)=q\}/\bR.$$
The cohomology of the complex $({\rm CM}^k(f,g), d)$ is called the Morse cohomology for the pair $(f,g)$ which we denote by ${\rm HM}^*(f,g)$. It can be shown that ${\rm HM}^*(f,g)$ is isomorphic to the singular cohomology of $H^*(W)$ over $\bZ_2$, and it does not depend on the Morse-Smale pair $(f,g)$.

Let $f$ be a function adapted to $W$ and $H\in\cH$. We recall the construction of the Piunikhin-Salamon-Schwarz (PSS) homomorphism from ${\rm HM}^*(f,g)$ to $\HF^*(H)$ by counting isolated spiked Floer strips. Let $p\in W$ be a critical point of $f$ and let $x\in\cP(H)$. Consider $(J^s)_{s\in\bR}\subset \cJ$ be a family of regular complex structures.  Further, let $\chi:\bR\ra[0,1]$ be a smooth function that satisfies $\chi(s) = 1$ if $s\leq 0$, $\chi(s) = 0$ if $s \geq 1$, and $\chi'(s) \leq 0$ for all $s$. Then we consider the moduli space $\cM(p,x;H,J,\chi,f,g)$ of the pairs of maps 
$$z:[0,\infty)\ra W,\quad u:\bR\times [0,1]\ra \widehat{W},$$
which satisfy
$$\begin{cases}
\frac{dz}{dt}=-\nabla f(z(t)),\\
\partial_s u+J^s_{t}\big(\partial_t u-\chi(s)X_{H}(u)\big)=0\;\hbox{with}\;E(u)<\infty,\\
z(+\infty)=p,\; u(-\infty,t)=x(t),\\
z(0)= u(+\infty).
\end{cases}$$
Up to generic choices of $(f,g,J)$, the space $\cM(p,x;H,J,\chi,f,g)$ is a smooth manifold with dimension $\mu(x)-\ind_f(p)$.

We define the morphism $\psi^H_f:{\rm CM}^k(f,g)\to \CF^k(H)$ on generators by
$$
\psi^H_f(p)=\sum_{\mu(x)=\ind_f(p)}\sharp_{\bZ_2}\cM(p,x;J,H,\chi,f,g)\cdot x
$$
and extend this map by linearity over $\bZ_2$. This extended map is a chain map and induces the PSS-type homomorphism
$$\psi_{f}^H:{\rm HM}^k(f,g)\longrightarrow \HF^k(H).$$
Since $H^*(W)\cong {\rm HM}^*(f,g)$, we can simplify the notation and denote the corresponding homomorphism from $H^*(W)$ to $\HF^*(H)$ as $PSS_H$. 

\medskip

For the relative case, we work on the functions that adapted to $L$ similarly, which means that the following conditions hold:
\begin{enumerate}
    \item No critical points of $f$ occur in $\widehat{L}\backslash L\cong\partial L\times(1,\infty)$;
    \item The gradient vector field $\nabla f$ of $f$ with respect to $g$ points outward along $\partial L$;
    \item  $f|_L$ is a $C^2$-smal Morse function and $(f,g)$ is a Morse-Smale pair.
\end{enumerate}
We can also define the the Morse complex ${\rm CM}^k(L,f,g)=\bZ_2\left<\text{Crit}_k(f)\right>$ and the differential $d$.
The Morse cohomology ${\rm HM}^*(L,f,g)$ is isomorphic to the singular cohomology of $H^*(L)$ over $\bZ_2$, and it does not depend on the Morse-Smale pair $(f,g)$.

To define the Lagrangian PSS map, we consider the moduli space $\cM(p,x;H,J,\chi,f,g)$ of the pairs of maps 
$$z:[0,\infty)\longrightarrow L,\quad u:\bR\times [0,1]\longrightarrow \widehat{W}$$
which satisfy
$$\begin{cases}
\frac{dz}{dt}=-\nabla f(z(t)),\\
\partial_s u+J^s_{t}\big(\partial_t u-\chi(s)X_{H}(u)\big)=0\;\hbox{with}\;E(u)<\infty,\\
z(+\infty)=p,\; u(-\infty,t)=x(t),\\
u(s,0), u(s,1)\in L,\\
z(0)= u(+\infty).
\end{cases}$$
Up to generic choices of $(f,g,J)$, the space $\cM(p,x;H,J,\chi,f,g)$ is a smooth manifold with dimension $\mu(x)-\ind_f(p)$.
Then the morphism $\psi^H_f:{\rm CM}^k(L,f,g)\to \CW^k(L,H)$ given by
$$
\psi^H_f(p)=\sum_{\mu(x)=\ind_f(p)}\sharp_{\bZ_2}\cM(p,x;J,H,\chi,f,g)\cdot x
$$
induces a PSS-type homomorphism
$$\psi_{f}^H:{\rm HM}^k(L,f,g)\longrightarrow \HW^k(L,H).$$Since $H^*(L)\cong HM^*(L,f,g)$, we can simplify the notation and denote the corresponding homomorphism from $H^*(L)$ to $\HW^*(L,H)$ as $PSS_H$. 

\section{Proofs}\label{sec-3}
\subsection{Proof of Theorem~\ref{betaleqgamma}}
The proof of the (absolute) Hamiltonian case closely follows Section 4 in \cite{KS21}, and the proof of the (relative) Lagrangian case is slightly different, so we only prove the Lagrangian case.

Let $H$ be an admissible Hamiltonian with small slope and $C^2$-small closed to a compactly supported Hamiltonian when restricted to $W$. Fix a homotopy class $\fc\in\pi_0(\cP(\widehat{L}_0, \widehat{L}_1))$. For homogeneous cocycles $x\in \CW^k_{\fc}(L_0, H)$, $y\in \CW^l_\fc(L_0, \overline{H})$, with $a=\cA(x), b=\cA(y)$, we can define the following filtered maps due to (\ref{x-product}), 
\begin{equation}\label{eq-continuation}
\begin{split}
x \ast-:&\CW^*_\fc(L_0, L_1; F)\ra \CW^{*+k}_\fc(L_0, L_1; H\sharp F)[a]\\
y\ast-:&\CW^*_\fc(L_0, L_1; G)\ra \CW^{*+l}_\fc(L_0, L_1; \overline{H}\sharp G)[b].
\end{split}
\end{equation}
Moreover, the following useful lemma holds by the relation (\ref{mu3-rel}) in Section \ref{section-product}. 
\begin{lemma}[cf. Lemma~40 in \cite{KS21}]\label{lemma-htp}
     For any $\varepsilon>0$, we can choose the Floer perturbation data such that the following chain maps
    $$y\ast(x\ast-):\CW^*_\fc(L_0, L_1; F)\ra \CW^{*+k+l}_\fc(L_0, L_1; F)[a+b+\varepsilon]$$
    $$((y\ast x)\ast-):\CW^*_\fc(L_0, L_1; F)\ra \CW^{*+k+l}_\fc(L_0, L_1; F)[a+b+\varepsilon]$$
    are well-defined, filtered, and moreover, filtered chain homotopic.
\end{lemma}
\begin{proof}
    For any $z\in \CW^*_\fc(L_0, L_1; F)$, we have $d(\mu_3(y,x,z))+\mu_3(dy,x,z)+\mu_3(y,dx,z)+\mu_3(y,x,dz)=(y\ast x)\ast z+y\ast(x\ast z).$ As we work on $\bZ_2$, we have 
    \begin{align*}
       (y\ast(x\ast-))-((y\ast x)\ast-) =&d\mu_3(y,x,-)+\mu_3(dy,x,-)+\mu_3(y,dx,-)+\mu_3(y,x,d-)\\
       =&d\circ \mu_3(y,x,-)+\mu_3(y,x,-)\circ d 
    \end{align*}    
    As $\ell_{L_0, L_1; F}(\mu(y,x,z))\leq\ell_{L_0, L_1;F}(z)+a+b+\varepsilon$ for all $z$ and both $(y\ast(x\ast-))$ and $((y\ast x)\ast-)$ preserve filtrations, $(y\ast(x\ast-))$ and $((y\ast x)\ast-)$ are filtered chain homotopic.
\end{proof}

Back to the proof of Theorem \ref{betaleqgamma}, let $H$ be an admissible Hamiltonian with small slope and $C^2$-small close to a compactly supported Hamiltonian when restricted to $W$. Fix a homotopy class $\fc\in\pi_0(\cP(\widehat{L}_0, \widehat{L}_1))$. We  choose  homogeneous cocycles $x\in \CW^0(L_0, H)$, $y\in \CW^0(L_0, \overline{H})$ such that
\[
[x]=PSS_{H}(\mathds{1}_{L_0}),\, [y]=PSS_{\overline{H}}(\mathds{1}_{L_0}),
\]
and 
\[
\cA(x)=\ell(\mathds{1}_{L_0}, H),\, \cA(y)=\ell(\mathds{1}_{L_0}, \overline{H}).
\]
Let $a=\ell(\mathds{1}_{L_0}, H)$ and $b=\ell(\mathds{1}_{L_0}, \overline{H})$. The maps in \eqref{eq-continuation} induce morphisms at the filtered cohomology level, that is, for all $t\in\bR$, we have
\begin{align*}
    [x\ast-]: & \HW^*_{\fc,<t}(L_0, L_1; 0)\ra \HW^*_{\fc,<t+a}(L_0, L_1; H),\\
    [y\ast-]: & \HW^*_{\fc,<t}(L_0, L_1; H)\ra \HW^*_{\fc,<t+b}(L_0, L_1; 0).
\end{align*}
Here we take $F=0$ and $G=H$.
By Lemma \ref{lemma-htp}, for any small $\varepsilon>0$,
\[   
    [(x\ast y)\ast-]=[x\ast-]\circ[y\ast-]:\HW^*_{\fc,<t}(L_0, L_1; H)\ra \HW^*_{\fc,<t+a+b+\varepsilon}(L_0, L_1; H).
\]
We note that  $[x\ast y]=PSS_0(\mathds{1}_{L_0})$; therefore,
\[
   [x\ast-]\circ[y\ast-]=\iota_{a+b+\varepsilon}:\HW^*_{\fc, <t}(L_0, L_1; H)\ra \HW^*_{\fc,<t+a+b+\varepsilon}(L_0, L_1; H)
\]
where $\iota_{a+b+\varepsilon}$ is induced by $\iota_{a+b+\varepsilon}:\CW^*_{\fc, <t}(L_0, L_1; H)\ra \CW^*_{\fc, <t+a+b+\varepsilon}(L_0, L_1; H)$. Since $\iota_{a+b+\varepsilon}$ factors through the zero complex $\CW^*_{\fc}(L_0,L_1;0)$, we have $\iota_{a+b+\varepsilon}=0$ for every $t\in\bR$. This implies that, for any $\varepsilon>0$,
\[
\beta_{\fc}(L_0,L_1;H)\leq a+b+\varepsilon.
\]
Hence, $\beta_{\fc}(L_0,L_1;H)\leq a+b$.
Noting that $\gamma(L_0;H)=\ell(\mathds{1}_{L_0}, H)+\ell(\mathds{1}_{L_0}, \overline{H})=a+b$,  we obtain
 \begin{equation}\label{eq-bd-gamma}
 \beta_{\fc}(L_0, L_1; H)\leq \gamma(L_0; H).
 \end{equation}
 Furthermore, the slope of $H$ outside of $W$ does not affect the spectral norm and boundary depth, as noted in Step~1 in the Proof of Lemma~20 in \cite{Gon24}. Thus \eqref{eq-bd-gamma} holds for $H\in C^{\infty}_{\fc}([0,1]\times W)$.

\subsection{Proof of Theorem~\ref{Ham-infinite-flat}} \label{ssec-proof-ThmA} The proof closely follows the proof of Theorem~1.1 in \cite{Ush13} with necessary modifications to deal with our current situation - manifolds with boundaries. Fix $H:D^*_gN\ra\bR$ by $H(q,p)=2|p|-1/2$. For $f\in C^\infty_c(0,1)$, define 
$$\minmax f\coloneqq\inf\{f(s)\mid s \text{ is a local maximum of }f\},$$
where a ``local maximum'' need not be strict. Since $f$ is compactly 
supported, then $\minmax f\leq 0$. We also denote $f\circ H$ the smooth extension which extends smoothly by zero to $H^{-1}(\bR\backslash (0,1))$. The following result is the key estimate (used to produce large boundary depth $\beta$). It generalizes Theorem~5.6 in \cite{Ush13} (on closed manifolds) to manifolds with boundaries.
 
\begin{prop}\label{ham-estimate}
    Let $(N,g)$ be a Riemannian manifold that contains no non-constant contractible closed geodesics and $H:D^*_gN\ra\bR$ by $H(q,p)=2|p|-1/2$. Then we have
    $$\beta(f \circ H) = \beta(\phi^1_{f\circ H})\geq \minmax f-\min f$$
    for any function $f\in C^\infty_c(\bR)$ compactly supported in $(0,1)$.
\end{prop}

Assuming this proposition, consider a smooth function $f: \bR\ra [0,1]$ with the following properties (cf.~Example \ref{ex-Usher}):
\begin{itemize}
\item ${\rm{supp}}(f)=[\delta,1-\delta]$ for some fixed real number $\delta$ with $0<\delta<\frac{1}{4}$.
\item The only local extremum of $f|_{(\delta,1-\delta)}$ is a maximum, at $f(1/2)=1$.
\item $f''(s)<0$ if and only if $s\in (1/4, 3/4)$.
\item $f''(s)>0$ if and only if $s\in (\delta, 1/4)\cup (3/4,1-\delta)$.
\end{itemize}
As argued in Example \ref{ex-Usher}, for any  $a\in \bR^{\infty}$, the map $\psi: \bR^{\infty} \to {\rm Ham}(D_g^*N, \omega_{\rm can})$ defined by $\psi(a) = \phi^1_{\left(\sum_{i=0}^{\infty}a_if(2^{i+1}s-1)\right) \circ H}$ as in (\ref{hom-psi}) is a homomorphism.  Recall that in Example \ref{ex-Usher}, we denote $\sum_{i=0}^{\infty}a_if(2^{i+1}s-1)$ by $f_a$.

\medskip

Now we can give the proof of Theorem~\ref{Ham-infinite-flat}. For $a, b\in \bR^\infty$, we have 
\begin{align*}
    d_{\gamma}(\psi(a),\psi(b))=&\gamma(f_{a-b} \circ H)\\
    \geq&\beta(f_{a-b} \circ H)\\
    \geq& \minmax (f_{a-b}\circ H)-\min (f_{a-b}\circ H)\\
    \geq& -\min (f_{a-b}\circ H)
\end{align*}  
where the first inequality comes from Theorem~\ref{betaleqgamma} and second inequality comes from Proposition \ref{ham-estimate}. Similarly, $d_{\gamma}(\psi(a),\psi(b))\geq -\min (f_{b-a}\circ H).$ Thus 
$$d_{\gamma}(\psi(a),\psi(b))\geq\max\{-\min (f_{a-b}\circ H),-\min (f_{b-a}\circ H)\}=\|a-b\|_{\infty}.  $$
Moreover, by the inequality between the spectral norm and the Hofer norm in (\ref{gamma-hofer}), for any $a, b\in\bR^\infty$,
\begin{align*}
\begin{split}
d_{\gamma}(\psi(a),\psi(b))&\leq d_{\rm{Hofer}}(\psi(a),\psi(b))\\
&=\|\psi(a-b)\|_{\rm Hofer}\\
&=\|\phi^1_{f_{a-b}\circ H}\|_{\rm Hofer} \leq 2\|a-b\|_\infty.
\end{split}
\end{align*}
In this way, we finish the proof of Theorem~\ref{Ham-infinite-flat}. 

\medskip

In what follows, we will focus on the proof of Proposition \ref{ham-estimate}. As preparations, we need to recall the Morse theory on non-compact manifolds. For a non-compact manifold $M$, we can define the Morse cohomology with a coercive Morse function $g$, which means that for any $a\in\bR$, $M^a\coloneqq\{x\in M \mid g(x)\leq a\}$ is compact. This construction is  similar to the usual Morse theory, see \cite{Sch93}.

\begin{proof}[Proof of Proposition \ref{ham-estimate}] For $f \circ H$, we denote $\widetilde{f\circ H}$ as the admissible perturbation of $f\circ H$. Then for any $\delta>0$, by the proof of Theorem~4.5 in~\cite{Ush10}, (which only requires the different part to be compact, where our $D^*_gN$ satisfies),
 there exists a smooth Morse function $G:T^*N\ra\bR$,  such that 
    \begin{itemize}
    \item{} $G=\widetilde{f\circ H}$ on $T^*N\backslash D^*_gN$.
        \item{} $\|(G-f\circ H)|_{D^*_gN}\|_{C^0}<\delta$.
        \item{}  All contractible periodic orbits of $X_G$ with period at most $1$ in $D^*_gN$ are constant.
        \item{}  At each critical point $p$ of $G$ in $D^*_gN$, the Hessian of $G$ has operator norm less $\pi$.
    \end{itemize}
    Note that $\widetilde{f\circ H}$ is linear out of $D^*_gN$ with positive slope, $G$ is coercive and there is no periodic orbits of $X_G$ or critical points of $G$ out of $D^*_gN$. Consider a Morse complex ${\rm CM}^*(G)$, which is only generated by critical points in the interior of $D_g^*N$,  Morse cohomology group ${\rm HM}^*(\lambda G)$ is isomorphic to $H^*(T^*N;\bZ_2)\cong H^*(N;\bZ_2)$. 

Then for any $0<\lambda\ll 1$, Hamiltonian $\lambda G$ is $C^2$-small so that its Hamiltonian Floer complex coincides with its Morse complex. In particular, for all sufficiently small $\lambda\in(0,1]$, we have $\beta(\phi_{\lambda G}^1)=\beta_{\rm Morse}({\rm CM}^*(\lambda G))$. Moreover, as all contractible orbits of $X_{\lambda G}$ are constant for $\lambda\in(0,1]$, then the function
    $$\lambda\mapsto\dfrac{1}{\lambda}\beta(\phi_{\lambda G}^1)\quad\text{and}\quad\lambda\mapsto\dfrac{1}{\lambda}\beta_{\rm Morse}({\rm CM}^*(\lambda G))$$
    are both continuous from $(0,1]$ to the finite set $\{G(p)-G(q)\mid p,q\in{\rm{Crit}}(G)\}$. Since they coincide for all sufficiently small $\lambda$, they coincide for all $\lambda\in(0,1]$. In particular,  $\beta(\phi_{G}^1)=\beta_{\rm Morse}({\rm CM}^*(G))$. Then the same arguments for Lemma 5.7 and Corollary 5.8 in \cite{Ush13} 
     imply that for any regular value of $f$ such that $\min f<z<\minmax f\leq 0$ and $\min f+3\delta<z$, we have 
    $$\beta_{\rm Morse}({\rm CM}^*(G)) = \beta_{\rm Morse}(G)\geq z-\min f-2\delta.$$
 Then we have $\beta(\phi_G^1) \geq z-\min f-2\delta$. Since$\|\widetilde{f\circ H}-G\|_{C^0}\leq\delta$, we have $\beta(\phi_{f\circ H}^1)\geq z-\min f-3\delta$ for any $z<\minmax f$ and any $\delta>0$. Thus we complete the proof. \end{proof}

\subsection{Proof of Theorem~\ref{thm-Lag-flat}}\label{sec-lag}
Now we recall the results from \cite{Ush14}. Let $\gamma$ be a unit speed geodesic from $\gamma(0)=x_0$. A point $\gamma(a)$ is said to be conjugate to $x_0$ along $\gamma$, if $\exp_{\gamma(0)}:TN\ra N$ is singular at $a\gamma'(0)$. Fix a point $x_1\in N\backslash\{x_0\}$ that is not a conjugate point of $x_0$. Let
$$\cG(x_0,x_1)=\left\{\gamma:[0,1]\ra N\mid
    \gamma\text{ is a geodesic from $x_0$ to $x_1$}\right\}.$$
Obviously, one can restrict $\cG(x_0,x_1)$ to subset according to fixed homotopy class and index. For $\fc\in\pi_0(\mathcal{P}_N(x_0,x_1))$, where $\cP_N(x_0,x_1)$ denotes the space of smooth paths from $x_0$ to $x_1$ in $N$, and $l\in\bZ$, we define 
$$\cG_{\fc,l}(x_0,x_1)=\{\gamma\in\cG(x_0,x_1)\mid[\gamma]=\fc,\, {\rm{Morse}}(\gamma)=l\},$$
where ${\rm{Morse}}(\gamma)$, the Morse index of $\gamma$, is equal to the number of conjugate points along $\gamma$, counted with multiplicity. 

\medskip

To facilitate our calculation later, we need the following assumption:
\begin{assm}\label{assm}
There is a homotopy class $\fc\in \pi_0(\mathcal{P}_N(x_0,x_1))$ and an integer $k$ such that:
\begin{itemize} 
\item[(i)] $\cG_{\fc,k}(x_0,x_1)\neq\varnothing$.
\item[(ii)] $\cG_{\fc,k}(x_0,x_1)\cup \cG_{\fc,k+2}(x_0,x_1)$ is a finite set.
\item[(iii)] $\cG_{\fc,k-1}(x_0,x_1)\cup\cG_{\fc,k+1}(x_0,x_1)=\varnothing$
\item[(iv)] Either $n\neq 2$ or $k\neq 0$.
\end{itemize}
\end{assm}

\begin{remark} [Notations on homotopy classes] \label{rmk-homotopy-class} Here, let us clarify the notation $\mathfrak c$. By definition \eqref{eq-chain-complex}, $\mathfrak c$ in $\CW_{\fc}^*\left(F_{x_0},F_{x_1}; H_a\right)$ denotes a homotopy class in $\pi_0(\cP(F_{x_0},F_{x_1}))$, while Assumption \ref{assm} is about a homotopy class in $\pi_0(\cP_N(x_0,x_1))$. To match them, in what follows in this section, we will always regard points $x_i \in N$ as points in the zero section $0_N \subset T^*N$, then $\fc\in \pi_0(\mathcal{P}_N(x_0,x_1))$ will be regarded as an element in $\pi_0(\cP(F_{x_0},F_{x_1}))$. \end{remark} 
Before we give the proof of Theorem~\ref{thm-Lag-flat}, let us introduce some notations. Consider the Hamiltonian $H: T^*N \to \bR_{\geq 0}$ by $H(q,p) = |p|$ and again the carefully designed family of functions $f: \bR \to [0,1]$ satisfying conditions in the previous section. Now we require that $f$ satisfies two extra properties:
\begin{itemize}
\item $f''(s)<0$ if and only if $s\in(\frac{1}{4},\frac{3}{4})$.
\item $f''(s)>0$ if and only if $s\in(\delta,\frac{1}{4})\cup(\frac{3}{4},1-\delta)$.
\end{itemize}
The above properties are important in the proof of Proposition~\ref{prop-large-beta}.
Recall that the Hamiltonian $H_a = f_a \circ H$ parametrized by $a \in \bR^{\infty}$ where $f_a$ is the function constructed in Example \ref{ex-Usher}. We have the following proposition:
\begin{prop}[Proposition~4.3 in \cite{Ush14}]\label{prop-large-beta}
Under Assumption~\ref{assm}, there is a constant $C$ such that, for all $a\in\bR^\infty$,
\[
\beta_{\fc}(F_{x_0}, F_{x_1}; H_a)\geq \|a\|_{\infty}-C.
\]
\end{prop}
\begin{remark}\label{rmk-dual}
In \cite{Ush14}, Usher works with homology rather than cohomology. However, the boundary depths defined from a filtered chain complex and its dual filtered (co)chain complex are identical, as shown in \cite{UZ16}. Thus, we obtain the result regarding the boundary depth defined from a filtered cochain complex.
\end{remark}
Now we can give the proof of Theorem~\ref{thm-Lag-flat}.
By the Ham-invariance of the $\gamma$ and Theorem~\ref{betaleqgamma}, we have
\begin{align*}
    \delta_{\gamma}(\phi_{H_{a}}^1 (F_{x_0}),\phi_{H_{b}}^1(F_{x_0}))=\gamma(F_{x_0};H_{a-b})
    \geq\beta_{\fc}(F_{x_0}, F_{x_1};H_{a-b})\geq\|a-b\|_\infty-C.
\end{align*}
 Finally, similarly to the later part of the proof of Theorem \ref{Ham-infinite-flat} in Section \ref{ssec-proof-ThmA}, based on the inequality between the spectral norm and the Hofer norm as in (\ref{gamma-hofer}), for any $a, b\in\bR^\infty$, by Theorem~\ref{module-structure}, we have 
\begin{align*}
\delta_{\gamma}(\phi_{H_{a}}^1 (F_{x_0}),\phi_{H_{b}}^1(F_{x_0}))&\leq d_{\gamma}(\phi_{H_{a}}^1,\phi_{H_{b}}^1) \leq d_{\rm{Hofer}}(\phi_{H_{a}}^1,\phi_{H_{b}}^1)\leq 2\|a-b\|_\infty.
\end{align*}
This concludes the proof.

\subsection{Proof of Theorem~\ref{module-structure}} \label{sec-mod} 
Consider a $2$-dimensional disk $\bD$ with two boundary points $z_0, z_1$ and an interior point $x$ removed. We consider the positive end near $z_1$, which is a holomorphic embedding $\kappa_1:\bR_+\times[0,1]\ra\bD$ satisfying $$\kappa^{-1}_1(\partial\bD)=\bR_{+}\times\{0,1\}\,\text{ and }\lim_{s\ra+\infty}\kappa_1(s,\cdot)=z_1.$$ 
Near $z_0$, we equip a negative strip-like end, similarly to $z_1$.The positive end near $x$ which is a holomorphic embedding $\kappa_2:\bR_+\times [0,1]\ra\bD$ satisfying 
$$\kappa^{-1}_2(\partial\bD)=\{0\}\times\{0,1\}, \,\,\,\,\mbox{and}\,\,\,\,\lim_{s\ra+\infty}\kappa_2(s,\cdot)=x.$$
Similarly, $\bD$ can be described as the following quotient of a strip with a slit: one consider the disjoint union 
$\bR \times [0,0.5] \sqcup \bR \times [0.5,1]$
and  identify $(s,0.5^-)$ with $(s,0.5^+)$ for every $s\leq 0$ and $(s,0.5^+)$ with $(s,1)$ for every $s\geq 0$. See Figure~\ref{Fig-module} (cf.~Figure \ref{fig-1}). 

\begin{figure}[ht]
	\centering
\includegraphics[scale=1]{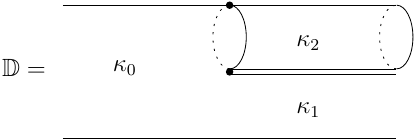}
\caption{The module structure.}\label{Fig-module}
\end{figure}

The setting here is similar to Section~\ref{section-product}. Denote by $K\in\Omega^1(\bD,C^{\infty}(\widehat{W},\bR))$ a Hamiltonian $1$-form, $X_K\in\Omega^1(\bD,\Gamma^\infty(T\widehat{W}))$ a Hamiltonian vector field associated with $K$. Let $J=\{J_z\}_{z\in\bD}$ be a family of almost complex structures parameterized by $\bD$ that satisfies $J_{\kappa_i(s,t)}=(J_i)_t$, $i=0,1,2$. Additionally, we require that $K|_{T\partial\bD}\equiv 0$, $\kappa_i^*K=(H_i)_t\otimes dt$, for $i=1,2$ and $\kappa_0^*K=(H_1\sharp H_2)_t\otimes dt$. Consider the inhomogeneous $\bar{\partial}$-equation
\begin{equation}\label{module}
    \begin{cases}
    (du-X_K)^{0,1}=0,\\
    u:\bD\ra\widehat{W},\,u(\partial\bD)\subset\widehat{L},\\
    \lim_{s\ra\pm\infty}u(\kappa_i(s,\cdot))=z_i,\,i=0,1,\\
    \lim_{s\ra\infty}u(\kappa_2(s,\cdot))=x
\end{cases}
\end{equation}
where  each $z_i$ is a chord of the Hamiltonian $H_i$ with respect to $\widehat{L}$ and $x$ is the Hamiltonian orbit of $H_2$. We denote the moduli space $\cM(K,J;z_0,z_1,x)$ is the space of the solutions $u$ to \eqref{module} with finite energy.  Define the chain map 
$$\ast:\CW^*(L,H_1)\otimes \CF^*(H_2)\ra \CW^*(L; H_1\sharp H_2)$$
$$z_1\ast x=\sum_{|z_0|=|z_1|+|x|}\sharp_{\bZ_2}\cM(K,J;z_0,z_1,x)z_0.$$
It induces a well-defined morphism on the filtered wrapped Floer cohomologies, 
$$\ast:\HW^*_{<a}(L; H_1)\otimes \HF^*_{<b}(H_2)\ra \HW^*_{<a+b}(L; H_1\sharp H_2),$$
for any $a,b\in\bR$. Then by the standard compactness and gluing arguments, we have the following commutative diagram (where the upper horizontal arrow is given by counting certain Morse flow trajectories when viewing both $H^*(W)$ and $H^*(L)$ as Morse cohomologies):
$$
\begin{tikzcd}
H^*(W)\otimes H^*(L) \arrow[r] \arrow[d, "PSS_K\otimes PSS_H"'] & H^*(L) \arrow[d, "PSS_{H\sharp K}"] \\
{\HF^*(K)\otimes \HW^*(L;H)} \arrow[r, "\ast"]                    & {\HW^*(L;H\sharp K)}                
\end{tikzcd}
$$
From the module structure
$\ast:\HW^*_{<a}(L,H)\otimes \HF^*_{<b}(K)\ra \HW^*_{<a+b}(L,H\sharp K),$
we can deduce the following inequality  on spectral invariants:
$\ell(\alpha\cdot\beta;H\sharp K)\leq \ell(\beta;H)+c(\alpha;K)$, as requested.

\bibliographystyle{amsplain}
\bibliography{biblio}

@unpublished{FZ25,
	author = {Feng, Qi and Zhang, Jun},
	note = {In progress},
	title = {Spectrally-large scale geometry via set-heaviness}}

@article{Gon24,
	author = {Gong, Wenmin},
	doi = {10.1016/j.geomphys.2024.105223},
	fjournal = {Journal of Geometry and Physics},
	issn = {0393-0440,1879-1662},
	journal = {J. Geom. Phys.},
	mrclass = {53D40 (53D12)},
	mrnumber = {4746983},
	pages = {Paper No. 105223, 41},
	title = {The unbounded {L}agrangian spectral norm and wrapped {F}loer cohomology},
	url = {https://doi.org/10.1016/j.geomphys.2024.105223},
	volume = {202},
	year = {2024}}

@book{Sch93,
	author = {Schwarz, Matthias},
	doi = {10.1007/978-3-0348-8577-5},
	isbn = {3-7643-2904-1},
	mrclass = {58E05 (55N35 57R70)},
	mrnumber = {1239174},
	mrreviewer = {Daniel\ M.\ Burns, Jr.},
	pages = {x+235},
	publisher = {Birkh\"auser Verlag, Basel},
	series = {Progress in Mathematics},
	title = {Morse homology},
	url = {https://doi.org/10.1007/978-3-0348-8577-5},
	volume = {111},
	year = {1993}}

@book{Ban97,
	author = {Banyaga, Augustin},
	doi = {10.1007/978-1-4757-6800-8},
	isbn = {0-7923-4475-8},
	mrclass = {22E65 (58B25 58D05)},
	mrnumber = {1445290},
	mrreviewer = {Vladimir\ Pestov},
	pages = {xii+197},
	publisher = {Kluwer Academic Publishers Group, Dordrecht},
	series = {Mathematics and its Applications},
	title = {The structure of classical diffeomorphism groups},
	url = {https://doi.org/10.1007/978-1-4757-6800-8},
	volume = {400},
	year = {1997}}

@book{MS12,
	author = {McDuff, Dusa and Salamon, Dietmar},
	edition = {Second},
	isbn = {978-0-8218-8746-2},
	mrclass = {53D45 (32Q65 53D35)},
	mrnumber = {2954391},
	mrreviewer = {Mark\ Alan\ Branson},
	pages = {xiv+726},
	publisher = {American Mathematical Society, Providence, RI},
	series = {American Mathematical Society Colloquium Publications},
	title = {{$J$}-holomorphic curves and symplectic topology},
	volume = {52},
	year = {2012}}

@article{RS93,
	author = {Robbin, Joel and Salamon, Dietmar},
	doi = {10.1016/0040-9383(93)90052-W},
	fjournal = {Topology. An International Journal of Mathematics},
	issn = {0040-9383},
	journal = {Topology},
	mrclass = {58F05 (58E05)},
	mrnumber = {1241874},
	mrreviewer = {J.\ S.\ Joel},
	number = {4},
	pages = {827--844},
	title = {The {M}aslov index for paths},
	url = {https://doi.org/10.1016/0040-9383(93)90052-W},
	volume = {32},
	year = {1993}}

@article{BK22,
	author = {Benedetti, Gabriele and Kang, Jungsoo},
	doi = {10.1007/s11784-022-00963-8},
	fjournal = {Journal of Fixed Point Theory and Applications},
	issn = {1661-7738,1661-7746},
	journal = {J. Fixed Point Theory Appl.},
	mrclass = {53D40 (37J46)},
	mrnumber = {4423620},
	mrreviewer = {Pawe\l \ Urbanski},
	number = {2},
	pages = {Paper No. 44, 32},
	title = {Relative {H}ofer-{Z}ehnder capacity and positive symplectic homology},
	url = {https://doi.org/10.1007/s11784-022-00963-8},
	volume = {24},
	year = {2022}}

@article{Che11,
	author = {Chekanov, Yu. V.},
	doi = {10.1007/PL00004814},
	fjournal = {Mathematische Zeitschrift},
	issn = {0025-5874,1432-1823},
	journal = {Math. Z.},
	mrclass = {53D12 (53C60 53D05)},
	mrnumber = {1774099},
	mrreviewer = {Darko\ Milinkovi\'{c}},
	number = {3},
	pages = {605--619},
	title = {Invariant {F}insler metrics on the space of {L}agrangian embeddings},
	url = {https://doi.org/10.1007/PL00004814},
	volume = {234},
	year = {2000}}

@article{FS07,
	author = {Frauenfelder, Urs and Schlenk, Felix},
	doi = {10.1007/s11856-007-0037-3},
	fjournal = {Israel Journal of Mathematics},
	issn = {0021-2172,1565-8511},
	journal = {Israel J. Math.},
	mrclass = {53D40 (37J05 37J45)},
	mrnumber = {2342472},
	mrreviewer = {Michael\ J.\ Usher},
	pages = {1--56},
	title = {Hamiltonian dynamics on convex symplectic manifolds},
	url = {https://doi.org/10.1007/s11856-007-0037-3},
	volume = {159},
	year = {2007}}

@article{Hof90,
	author = {Hofer, Helmut},
	doi = {10.1017/S0308210500024549},
	fjournal = {Proceedings of the Royal Society of Edinburgh. Section A. Mathematics},
	issn = {0308-2105,1473-7124},
	journal = {Proc. Roy. Soc. Edinburgh Sect. A},
	mrclass = {58F05 (70H05)},
	mrnumber = {1059642},
	mrreviewer = {Yong-Geun\ Oh},
	number = {1-2},
	pages = {25--38},
	title = {On the topological properties of symplectic maps},
	url = {https://doi.org/10.1017/S0308210500024549},
	volume = {115},
	year = {1990}}

@article{KS21,
	author = {Kislev, Asaf and Shelukhin, Egor},
	doi = {10.2140/gt.2021.25.3257},
	fjournal = {Geometry \& Topology},
	issn = {1465-3060,1364-0380},
	journal = {Geom. Topol.},
	mrclass = {57R17 (53D12 53D40 55N31)},
	mrnumber = {4372632},
	mrreviewer = {Jun\ Zhang},
	number = {7},
	pages = {3257--3350},
	title = {Bounds on spectral norms and barcodes},
	url = {https://doi.org/10.2140/gt.2021.25.3257},
	volume = {25},
	year = {2021}}

@article{Lec08,
	author = {Leclercq, R\'{e}mi},
	doi = {10.3934/jmd.2008.2.249},
	fjournal = {Journal of Modern Dynamics},
	issn = {1930-5311,1930-532X},
	journal = {J. Mod. Dyn.},
	mrclass = {57R17 (53D40 55T10)},
	mrnumber = {2383268},
	number = {2},
	pages = {249--286},
	title = {Spectral invariants in {L}agrangian {F}loer theory},
	url = {https://doi.org/10.3934/jmd.2008.2.249},
	volume = {2},
	year = {2008}}

@article{LZ18,
	author = {Leclercq, R\'{e}mi and Zapolsky, Frol},
	doi = {10.1142/S1793525318500267},
	fjournal = {Journal of Topology and Analysis},
	issn = {1793-5253,1793-7167},
	journal = {J. Topol. Anal.},
	mrclass = {57R17 (53D12 53D40)},
	mrnumber = {3850107},
	mrreviewer = {Roman\ Golovko},
	number = {3},
	pages = {627--700},
	title = {Spectral invariants for monotone {L}agrangians},
	url = {https://doi.org/10.1142/S1793525318500267},
	volume = {10},
	year = {2018}}

@unpublished{Mai22,
	author = {Mailhot, Pierre-Alexandre},
	note = {arXiv preprint, arXiv:2205.04618},
	title = {The spectral diameter of a {L}iouville domain},
	year = {2022}}

@article{Mil01,
	author = {Milinkovi\'{c}, Darko},
	doi = {10.1090/S0002-9939-00-05851-2},
	fjournal = {Proceedings of the American Mathematical Society},
	issn = {0002-9939,1088-6826},
	journal = {Proc. Amer. Math. Soc.},
	mrclass = {53D40 (57R57 58E05)},
	mrnumber = {1814118},
	mrreviewer = {David\ E.\ Hurtubise},
	number = {6},
	pages = {1843--1851},
	title = {Geodesics on the space of {L}agrangian submanifolds in cotangent bundles},
	url = {https://doi.org/10.1090/S0002-9939-00-05851-2},
	volume = {129},
	year = {2001}}

@article{PS23,
	author = {Polterovich, Leonid and Shelukhin, Egor},
	doi = {10.1112/s0010437x23007455},
	fjournal = {Compositio Mathematica},
	issn = {0010-437X,1570-5846},
	journal = {Compos. Math.},
	mrclass = {53D12 (37J06 53D05)},
	mrnumber = {4651502},
	number = {12},
	pages = {2483--2520},
	title = {Lagrangian configurations and {H}amiltonian maps},
	url = {https://doi.org/10.1112/s0010437x23007455},
	volume = {159},
	year = {2023}}

@article{Rit13,
	author = {Ritter, Alexander F.},
	doi = {10.1112/jtopol/jts038},
	fjournal = {Journal of Topology},
	issn = {1753-8416,1753-8424},
	journal = {J. Topol.},
	mrclass = {53D40 (57R56 81T70)},
	mrnumber = {3065181},
	mrreviewer = {David\ E.\ Hurtubise},
	number = {2},
	pages = {391--489},
	title = {Topological quantum field theory structure on symplectic cohomology},
	url = {https://doi.org/10.1112/jtopol/jts038},
	volume = {6},
	year = {2013}}

@article{She22b,
	author = {Shelukhin, Egor},
	doi = {10.1007/s00222-022-01124-x},
	fjournal = {Inventiones Mathematicae},
	issn = {0020-9910,1432-1297},
	journal = {Invent. Math.},
	mrclass = {53D40 (57R17)},
	mrnumber = {4480149},
	mrreviewer = {Olga\ Bernardi},
	number = {1},
	pages = {321--373},
	title = {Viterbo conjecture for {Z}oll symmetric spaces},
	url = {https://doi.org/10.1007/s00222-022-01124-x},
	volume = {230},
	year = {2022}}

@article{Sug16,
	author = {Sugimoto, Yoshihiro},
	doi = {10.1007/s11784-016-0287-y},
	fjournal = {Journal of Fixed Point Theory and Applications},
	issn = {1661-7738,1661-7746},
	journal = {J. Fixed Point Theory Appl.},
	mrclass = {53D40 (53D12)},
	mrnumber = {3551405},
	mrreviewer = {Darko\ Milinkovi\'{c}},
	number = {3},
	pages = {547--567},
	title = {Hofer's metric on the space of {L}agrangian submanifolds and wrapped {F}loer homology},
	url = {https://doi.org/10.1007/s11784-016-0287-y},
	volume = {18},
	year = {2016}}

@article{Ush10,
	author = {Usher, Michael},
	doi = {10.1142/S0219199710003889},
	fjournal = {Communications in Contemporary Mathematics},
	issn = {0219-1997,1793-6683},
	journal = {Commun. Contemp. Math.},
	mrclass = {53D35 (53D40 53D45)},
	mrnumber = {2661273},
	mrreviewer = {David\ E.\ Hurtubise},
	number = {3},
	pages = {457--473},
	title = {The sharp energy-capacity inequality},
	url = {https://doi.org/10.1142/S0219199710003889},
	volume = {12},
	year = {2010}}

@article{Ush13,
	author = {Usher, Michael},
	doi = {10.24033/asens.2185},
	fjournal = {Annales Scientifiques de l'\'{E}cole Normale Sup\'{e}rieure. Quatri\`eme S\'{e}rie},
	issn = {0012-9593,1873-2151},
	journal = {Ann. Sci. \'{E}c. Norm. Sup\'{e}r. (4)},
	mrclass = {53D40 (53D12 53D35)},
	mrnumber = {3087390},
	mrreviewer = {David\ E.\ Hurtubise},
	number = {1},
	pages = {57--128},
	title = {Hofer's metrics and boundary depth},
	url = {https://doi.org/10.24033/asens.2185},
	volume = {46},
	year = {2013}}

@article{Ush14,
	author = {Usher, Michael},
	fjournal = {The Journal of Symplectic Geometry},
	issn = {1527-5256,1540-2347},
	journal = {J. Symplectic Geom.},
	mrclass = {53D05 (37D40 53C22 53D40)},
	mrnumber = {3248671},
	mrreviewer = {Chun-Gen\ Liu},
	number = {3},
	pages = {619--656},
	title = {Hofer geometry and cotangent fibers},
	url = {http://projecteuclid.org/euclid.jsg/1409319463},
	volume = {12},
	year = {2014}}

@article{UZ16,
	author = {Usher, Michael and Zhang, Jun},
	doi = {10.2140/gt.2016.20.3333},
	fjournal = {Geometry \& Topology},
	issn = {1465-3060,1364-0380},
	journal = {Geom. Topol.},
	mrclass = {53D40 (55U15)},
	mrnumber = {3590354},
	mrreviewer = {Sonja\ Hohloch},
	number = {6},
	pages = {3333--3430},
	title = {Persistent homology and {F}loer-{N}ovikov theory},
	url = {https://doi.org/10.2140/gt.2016.20.3333},
	volume = {20},
	year = {2016}}

@article{Vit92,
	author = {Viterbo, Claude},
	doi = {10.1007/BF01444643},
	fjournal = {Mathematische Annalen},
	issn = {0025-5831,1432-1807},
	journal = {Math. Ann.},
	mrclass = {58F05 (53C15 53C23 57R50 58E05)},
	mrnumber = {1157321},
	mrreviewer = {Nikolai\ K.\ Smolentsev},
	number = {4},
	pages = {685--710},
	title = {Symplectic topology as the geometry of generating functions},
	url = {https://doi.org/10.1007/BF01444643},
	volume = {292},
	year = {1992}}

@article{Vit99,
	author = {Viterbo, Claude},
	doi = {10.1007/s000390050106},
	fjournal = {Geometric and Functional Analysis},
	issn = {1016-443X,1420-8970},
	journal = {Geom. Funct. Anal.},
	mrclass = {53D40 (57R17 57R58)},
	mrnumber = {1726235},
	mrreviewer = {David\ E.\ Hurtubise},
	number = {5},
	pages = {985--1033},
	title = {Functors and computations in {F}loer homology with applications. {I}},
	url = {https://doi.org/10.1007/s000390050106},
	volume = {9},
	year = {1999}}

@unpublished{Vit08,
	author = {Viterbo, Claude},
	note = {arXiv preprint, arXiv:0801.0206},
	title = {Symplectic homogenization},
	year = {2008}}


\end{document}